\documentclass[11pt,reqno]{amsart}
\usepackage[english]{babel}
\usepackage{amscd}
\usepackage{amsmath,amsthm,amssymb,amsfonts}
\usepackage{mathrsfs}
\usepackage{graphicx}
\usepackage{fancyhdr}
\usepackage{stmaryrd}
\usepackage[colorinlistoftodos,prependcaption,textsize=tiny]{todonotes}
\usepackage{color}
\usepackage{amsmath}
\usepackage{amsfonts}
\usepackage{amssymb}
\usepackage{geometry}
\usepackage{esint}
\usepackage{mathtools}
\usepackage{chngcntr}
\usepackage{apptools}
\AtAppendix{\counterwithin{lemma}{section}}
\numberwithin{equation}{section}

\newtheorem{theorem}{Theorem}[section]
\newtheorem{proposition}{Proposition}[section]
\newtheorem{lemma}{Lemma}[section]
\newtheorem{remark}{Remark}[section]
\newtheorem{corollary}{Corollary}[section]
\newtheorem{definition}{Definition}[section]

\newcommand{\dif}{\,\mathrm{d}}

\def\R{\mathbb{R}}
\def\C{\mathcal{C}}
\def\E{\mathcal E}

\def \p{\partial}

\def \e {\varepsilon}
\def \I {\mathcal I}

\def \O {\Omega}

\def \N {\mathcal{N}}

\DeclareMathOperator{\dive}{div}

\DeclareMathOperator{\tr}{tr}
\DeclareMathOperator{\curl}{curl}
\DeclareMathOperator{\argmin}{argmin}

\makeatletter
\def\@seccntformat#1{\@ifundefined{#1@cntformat}%
   {\csname the#1\endcsname\quad}
   {\csname #1@cntformat\endcsname}
}
\makeatother

\author{R\'emy Rodiac, Pa\'ul Ubill\'us}
\address[R\'emy Rodiac]{Université Paris-Saclay, CNRS,  Laboratoire de mathématiques d'Orsay, 91405, Orsay, France}
\email{remy.rodiac@universite-paris-saclay.fr}
\address[Pa\'ul Ubill\'us]{Universite catholique de Louvain, Institut de Recherche en Mathématique et Physique (IRMP), Chemin du Cyclotron 2, 1348 Louvain-la-Neuve, Belgium}
\email{paul.ubillus@uclouvain.be}
\title[Renormalized energies]{Renormalized energies for unit-valued harmonic maps in multiply connected domains}
\date{}

\begin{document}
\begin{abstract}
In this article we derive the expression of \textit{renormalized energies} for unit-valued harmonic maps defined on a smooth bounded domain in \(\R^2\) whose boundary has several connected components. The notion of renormalized energies was introduced by Bethuel-Brezis-H\'elein in order to describe the position of limiting Ginzburg-Landau vortices in simply connected domains. We show here, how a non-trivial topology of the domain modifies the expression of the renormalized energies. We treat the case of Dirichlet boundary conditions and Neumann boundary conditions as well.
\end{abstract}

\maketitle

\section{Introduction}

The motivation for introducing the notion of renormalized energy of unit-valued harmonic maps comes from a topological obstruction. As observed by Bethuel-Brezis-H\'elein in their pioneering work \cite{BBH_1994}, if \(G\subset \R^2\) is a smooth bounded domain and \(g \in \C^1(\p G,\mathbb{S}^1)\), the space
\[H^1_g(G,\mathbb{S}^1):= \{ u \in H^1(G,\mathbb{C}); \tr_{|\p G} u=g ,\ |u|=1 \text{ a.e.}\} \]
can be empty. In order to explain this, we introduce the definition of the topological degree. If \(\Gamma\) is a smooth simple closed curve and if \( g \in \C^1(\Gamma,\mathbb{S}^1)\), the topological degree of \(g\) is defined by
\begin{equation}\label{eq:degree}
\deg(g, \Gamma)=\frac{1}{2\pi}\int_{\Gamma} g \wedge \p_\tau g
\end{equation}
where \(\tau\) is the tangent vector to the curve, oriented anti-clockwise and the wedge product \( \wedge\) is defined by 
\begin{equation*}
a \wedge b =\frac{1}{2i} (\bar{a}b-a\bar{b})=a_1b_2-a_2b_1 \quad \text{ for } a=a_1+ia_2,\ b=b_1+ib_2 \in \mathbb{C}.
\end{equation*}
It can be shown that the topological degree is an integer (see e.g.\ \cite{Nirenberg_1974}). Furthermore the degree can be extended to functions \(g\) in \(H^{\frac12}(\Gamma,\mathbb{S}^1)\) by using formula \eqref{eq:degree}, where the product is understood in the sense of the \(H^{\frac12}-H^{-\frac12}\) duality. This remains integer-valued as was observed in  the appendix of \cite{Boutet-de-Monvel-Berthie_Georgescu_Purice_1991}, see also \cite{Brezis_Nirenberg_1995,Brezis_1997,Brezis_2006}.  In the rest of the paper, unless stated otherwise, \(G\) is a smooth bounded domain which is multiply connected, i.e., \(\pi_1(G)\neq \{0\}\) where \(\pi_1(G)\) is the fundamental group of \(G\). More precisely \(G= \widetilde{G}\setminus \cup_{l=1}^n \overline{\omega}_l\), where  \(n \in \mathbb{N}^*\) and \( \widetilde{G}, \omega_l, l=1,\dots,n\) are simply connected smooth bounded domains. We call \(\Gamma_0=\p \widetilde{G}\) and \(\Gamma_l=\p \omega_l\), \(l=1,\dots,n\). We fix a boundary data \(g\) on \(\p G\), that we assume to be \(\C^1\) for simplicity. Then, we recall

\begin{proposition}\label{prop:Sobolev_empty_or_not}
The space \(H^1_g(G,\mathbb{S}^1)\) is not empty if and only if \( \sum_{i=1}^n \deg( g, \Gamma_l)= \deg(g,\Gamma_0).\)
\end{proposition}

 If \(H^1_g(G,\mathbb{S}^1)=\emptyset\) there is no unit-valued harmonic map with trace \(g\), i.e., there is no critical point of the Dirichlet energy 
\[ E(u)=\frac12 \int_G |\nabla u|^2\]
in the space \(H^1_g(G,\mathbb{S}^1)\). We can then relax the problem of finding a unit-valued harmonic map with trace \(g\) by creating small holes in the domain. More precisely we consider \(k \in \mathbb{N}^*\), \(a_1, \dots,a_k \in G\), \(d_1,\dots,d_k\in \mathbb{Z}\) such that 
\begin{equation}\label{eq:relation_degrees}
\sum_{i=1}^k d_i+\sum_{l=1}^n \deg(g, \Gamma_l)= \deg(g,\Gamma_0). 
\end{equation}
For \(\rho\) small enough so that the balls \(\bar{B}_\rho(a_i)\) are disjoint and included in \(G\), we set 
\begin{equation}\label{eq:Omega_rho}
\O_\rho:= G \setminus \cup_{i=1}^k \bar{B}_\rho(a_i),
\end{equation}
\begin{equation}\label{eq:E_rho}
\E_{g,\rho}:=\{ u \in H_g^1(\O_\rho, \mathbb{S}^1); \tr_{|\p G} u =g; \ \deg(u, \p B_\rho(a_i))=d_i\},
\end{equation}
\begin{equation}\label{eq:min_w_d_rho}
W_{g}^\rho(\{a_i\},\{d_i\}):= \inf_{u \in \E_{g,\rho}} \frac12 \int_{\O_\rho} |\nabla u|^2.
\end{equation}
We can then study the asymptotic behaviour of \( W_{g}^\rho (\{a_i\},\{d_i\})\) as \(\rho \rightarrow 0\) and the convergence of minimizers for \(W_{g}^\rho\) (we will prove in Proposition \ref{prop:existence_Euler_Lagrange} that minimizers exist). When \(G\) is simply connected Bethuel-Brezis-H\'elein proved that
\begin{equation}\label{eq:renormalized_energy_abstract}
W_{g}(\{a_i\},\{d_i\}):=\lim_{\rho \rightarrow 0} \left( W_g^\rho(\{a_i\},\{d_i\})-\pi \left(\sum_{i=1}^k d_i^2 \right)|\log \rho| \right) <+\infty,
\end{equation}
and they gave an expression of \(W_g(\{a_i\},\{d_i\})\) in terms of Green functions with Neumann boundary condition, cf.\ Theorem I.7 in \cite{BBH_1994}. The quantity \(W_g(\{a_i\},\{d_i\})\) is called the \textit{renormalized energy} of the configurations \( (\{a_i\},\{d_i\}) \) (with Dirichlet boundary condition). In \cite{BBH_1994} the authors also related this renormalized energy to another way of relaxing the problem of finding unit-valued harmonic map with a given trace \(g\). They considered the Ginzburg-Landau energy 
\begin{equation}
E_\e(u)= \frac{1}{2}\int_G |\nabla u|^2+\frac{1}{4\e^2}\int_G (1-|u|^2)^2,
\end{equation}
defined in \(H^1_g:=\{  u \in H^1(G,\mathbb{C}); \tr_{|\p G} u=g \}\) and studied the asymptotic behaviour of a family of minimizers \( (u_\e)_\e\) of \(E_\e\) in \(H^1_g\). When \(G\) is star-shaped and \(\deg(g, \p G)=d \neq 0\), they proved that there exist \(d\) points \(a_1,\dots,a_d\) in \(G\), a singular harmonic map \(u_*\in \C^\infty(G \setminus \{a_1,\dots,a_d\}, \mathbb{S}^1)\) such that \(u_*\) has degree \(1\) around each \(a_i\), with \(u_{\e_p} \rightarrow u_* \) in \(\C^1(G \setminus \{a_1,\dots,a_d\})\), up to a subsequence \(\e_p \to 0\) and with the \(a_i\)'s which minimize the renormalized energy \(W_g(\{a_i\},\{d_i=1\})\). This was extended to simply connected domains in \cite{Struwe_1994}, \cite{delPino_Felmer_1998}. Recently, in \cite{Monteil_Rodiac_VanSchatfingen_2020a,Monteil_Rodiac_VanSchatfingen_2020b}, the authors obtained an analogous result where \(\mathbb{S}^1\) is replaced by an arbitrary smooth compact Riemannian manifold \(\mathcal{N}\) and without any assumption on the topology of \(G\). However, in \cite{Monteil_Rodiac_VanSchatfingen_2020a,Monteil_Rodiac_VanSchatfingen_2020b}, the renormalized energy is given by an abstract formula similar to \eqref{eq:renormalized_energy_abstract}. One of the goals of this article is to derive an explicit expression of this renormalized energy when \(\mathcal{N}=\mathbb{S}^1\) and \(G\) is multiply connected.

Another motivation for studying renormalized energies in multiply connected domains is to have a better understanding of the role of the topology in this problem. In recent works \cite{Ignat_Jerrard_2017,Ignat_Jerrard_2020}, Ignat and Jerrard studied a Ginzburg-Landau problem for tangent vector fields defined on smooth closed Riemannian surfaces. In this context, another topological obstruction to the existence of \(H^1\) unit-valued vector fields occurs. This is due to the \(H^1\) version of the Poincar\'e-Hopf theorem, which states that when the genus of the surface is not equal to \(1\) there is no continuous (nor \(H^1\)) vector field of unit norm on the surface. Ignat-Jerrard introduced a renormalized energy and proved that this is the \(\Gamma\)-limit at second order of the Ginzburg-Landau functional they considered. They also showed that, compared to the work \cite{BBH_1994}, new terms appear in the renormalized energy when the genus of the surface is not zero. These terms involve flux-integrals of a limiting singular harmonic map, they depend on the position and of the degrees of the singular points and are constrained to belong to a vorticity-dependent lattice. The topology of a surface is determined by its genus and the number of the connected components of its boundary. Thus, in this article, we are interested in the effect of the number of the connected components of the boundary on the renormalized energy rather than the effect of the genus. We find that, in this case too, new terms  appear and they can also be computed as flux-integrals. As a side remark, we point out that the Ginzburg-Landau energy is used in superconductivity, superfluidity and nonlinear optics. In physics, and in particular in electromagnetic, it is known that the topology of the domain has an effect on the existence of potentials and this can be at the origin of a new phenomenon like, for example, the Ahoronov-Bohm effect \cite{aharonov1959significance}.

We now introduce some definitions in order to state our main results. We call \(\Phi_0\)  the solution to 
\begin{equation}\label{eq:Phi_0}
\left\{
\begin{array}{rcll}
\Delta \Phi_0 &=& 2\pi \sum_{i=1}^kd_i \delta_{a_i} & \text{ in } G,\\
\p_\nu \Phi_0&=& g \wedge \p_\tau g & \text{ on } \p G, \\
\int_{\p G} \Phi_0&=&0,
\end{array}
\right.
\end{equation}
and \(R_0\) the regular part of \(\Phi_0\) given by
\begin{equation}
R_0(x)=\Phi_0(x)-\sum_{i=1}^k d_i \log |x-a_i|.
\end{equation}
We define \(\varphi_l\), \(l=1,\dots,n\) to be the solutions to
\begin{equation}\label{eq:def_varphi_l}
\left\{
\begin{array}{rcll}
\Delta \varphi_l &=&0 &\text{ in } G, \\
\varphi_l&=&1 &\text{ on } \Gamma_l, \\
\varphi_l&=&0 &\text{ on } \Gamma_m, \ m \neq l.
\end{array}
\right.
\end{equation}
For \(g\in \C^1(\p G,\mathbb{S}^1)\) and \(d_1,\dots,d_k\in \mathbb{Z}\) verifying \eqref{eq:relation_degrees} we introduce
\begin{multline}\label{eq:class_Dirichlet}
\I_{g,{d_i}}:=\Bigl\{v\in H^1(G,\mathbb{S}^1); \deg(v,\Gamma_l)=\deg(g,\Gamma_l), \ l=1,\dots,n,  \\
\ \deg (v,\Gamma_0)=\deg(g,\Gamma_0)-\sum_{i=1}^kd_i\Bigr\}
\end{multline}
and we call \(U_{g,d_i}\) a minimizer of \(\frac12 \int_G |\nabla v|^2\) for \(v\in \I_{g,{d_i}}\), i.e.,
\begin{equation}\label{eq:U_g_d}
\frac12 \int_G |\nabla U_{g,{d_i}}|^2=\min_{v \in \I_{g,{d_i}}} \frac12 \int_G |\nabla v|^2.
\end{equation}
We will obtain in the proof of Theorem \ref{th:main1} that such a minimizer exists and is unique up to a phase. 

\begin{theorem}\label{th:main1}
Let \(g\in \C^1(\p G,\mathbb{S}^1)\), \(a_1,\dots,a_k \in G\), \(d_1,\dots,d_k\in \mathbb{Z}\) satisfying \eqref{eq:relation_degrees}. There exists a unique minimizer \(u_\rho\) for the problem \eqref{eq:min_w_d_rho}. There exist a subsequence \(\rho_p \to 0\) and a map \(u_0\in W^{1,q}(G,\mathbb{S}^1)\) for every \(1\leq q<2\) such that, as \(p \to +\infty\), \(u_{\rho_p} \rightarrow u_0\) in \(\C^m_{\text{loc}}(G\setminus \{a_1,\dots,a_k\})\) for all \(m\in \mathbb{N}\). The map \(u_0\) satisfies
\begin{equation}
\left\{
\begin{array}{rcll}
-\Delta u_0&=&|\nabla u_0|^2u_0 &\text{ in } G \setminus \{a_1,\dots,a_k\},\\
u_0&=& g &\text{ on } \p G,
\end{array}
\right.
\end{equation}
and \(u_0\) can be written as 
\begin{equation}\label{eq:canocnical_harmonic_map}
u_0=\prod_{i=1}^k \left(\frac{x-a_i}{|x-a_i|} \right)^{d_i} U_{g,{d_i}} e^{i \psi_g}
\end{equation}
where \(\psi_g\) is a harmonic function in \(G\) and \(U_{g,d_i}\) satisfies \eqref{eq:U_g_d}.
Furthermore,
\begin{equation}
W_g^\rho(\{a_i\},\{d_i\})=\pi \left( \sum_{i=1}^k d_i^2\right) |\log \rho| +W_g(\{a_i\},\{d_i\})+o(1),
\end{equation}
with
\begin{multline}\label{eq:renormalized_Dirichlet}
W_g(\{a_i\},\{d_i\})= -\pi \sum_{i\neq j} d_id_j \log |a_i-a_j|+\frac12 \int_{\p G} \Phi_0(g \wedge \p_\tau g) -\pi \sum_{i=1}^k d_iR_0(a_i) \\
+ \sum_{l=1}^n \int_{\Gamma_l} \alpha_l \p_\tau \Phi_0+\frac12\int_G \left|\sum_{l=1}^n \alpha_l\nabla \varphi_l \right|^2,
\end{multline}
where \(\alpha_l=\alpha_l(g,\{a_i\},\{d_i\})\) are real constants.\\   Besides, there exist \(\theta_l=\theta_l(g,\{a_i\},\{d_i\})\in [-\pi,\pi[\) such that  \(\alpha_l=\theta_l+2\pi \mathbb{Z}\) and the coefficients \(\alpha_l\) are solutions to the linear system
\begin{equation}\label{eq:alpha_l_sol_linear_system}
\sum_{m=1}^n \alpha_m \int_{\Gamma_m} \p_\nu \varphi_l= \int_{\Gamma_l} u_0 \wedge \p_\nu u_0 \quad \text{ for } l=1,\dots,n.
\end{equation}
\end{theorem}

We are also interested in the renormalized energy with Neumann boundary conditions. Indeed, although there is no topological obstruction related to the degree in this case and the minimizers of the Ginzburg-Landau energy without any constraints on the boundary are constants of unit modulus, one can be interested in the asymptotic behaviour of critical points of the Ginzburg-Landau energy. Furthermore, the Neumann boundary conditions are the natural conditions when we consider a Ginzburg-Landau energy with magnetic field, see e.g.\ \cite{Sandier_Serfaty_2007}. Let us first define this renormalized energy: for \(G\) a smooth bounded domain, \(k \in \mathbb{N}^*\), \(a_1,\dots,a_k\in G\), \(d_1,\dots,d_k \in \mathbb{Z}\) and \(\rho\) sufficiently small, we define
\begin{equation}\label{eq:class_min_Neumann}
\E_{\N,\rho}:=\{ u \in H^1(\O_\rho,\mathbb{S}^1); \deg(u,\p B_\rho(a_i))=d_i, \ i= 1,\dots,k \}
\end{equation}
\begin{equation}\label{eq:min_w_N_rho}
W_\N^\rho(\{a_i\},\{d_i\}):=\inf_{u \in \E_{\N,\rho}} \frac12 \int_{\O_\rho} |\nabla u|^2.
\end{equation}
The renormalized energy with Neumann boundary condition is defined as
\begin{equation}\nonumber
W_\N(\{a_i\},\{d_i\}):= \lim_{\rho \to 0} \left[ W_\N^\rho(\{a_i\},\{d_i\})-\pi \left( \sum_{i=1}^k d_i^2\right)|\log \rho| \right].
\end{equation}
When \(G\) is simply connected, this quantity was shown to be finite in \cite{Lefter_Radulescu_1996} and an expression in terms of Green functions with homogeneous Dirichlet boundary conditions was derived in the same article. We will obtain a similar result when \(G\) is multiply connected. It was shown in \cite{Serfaty_2005} that, when \(G\) is simply connected, critical points of the Ginzburg-Landau equation with homogeneous Neumann boundary condition converge to critical points of the renormalized energy (with Neumann boundary condition). For the renormalized energy with magnetic field and with Neumann boundary conditions we refer to \cite{Spirn_2003}, \cite{Du_Lin_1997}, \cite{Rubinstein_1995}. In the case of a multiply connected domain, the renormalized energy was formally derived in \cite{delPino_Kowalczyk_Musso_2006} as the limit when \(\e\rightarrow 0\) of the Ginzburg-Landau energy of a suitable approximation of a solution to the Ginzburg-Landau equation with homogeneous Neumann boundary conditions.
 We again introduce some definitions, we call
\(\hat{G}_0\) the solution to
\begin{equation}\label{eq:Green_Dirichlet}
\left\{
\begin{array}{rcll}
\Delta \hat{G}_0&=& 2\pi \sum_{i=1}^k d_i\delta_{a_i} & \text{ in } G, \\
\hat{G}_0&=&0 & \text{ on } \p G,
\end{array}
\right.
\end{equation}
and we call \(\hat{R}_0\) the regular part of \(\hat{G}_0\), i.e., 
\begin{equation}\label{eq:regular_part_Dirichlet}
\hat{R}_0(x)=\hat{G}_0(x)-\sum_{i=1}^k d_i \log |x-a_i|.
\end{equation}
We also define \(\psi_\N\) the solution to
\begin{equation}\label{eq:phase_Neumann}
\left\{
\begin{array}{rcll}
\Delta \psi_\N &=&0 &\text{ in } G,\\
\p_\nu \psi_\N &=& -\sum_{i=1}^k d_i \frac{(x-a_i)^\perp}{|x-a_i|^2} \cdot \nu & \text{ on } \p G,\\
\int_G \psi_\N&=&0.
\end{array}
\right.
\end{equation}
For \( \tilde{d}_0, \tilde{d}_1,\dots,\tilde{d}_n \in \mathbb{Z}\) verifying \(\sum_{l=1}^n \tilde{d}_l+\sum_{i=1}^kd_i=\tilde{d}_0\), we define
\begin{equation}\label{eq:class_I_Neumann}
\I_{\tilde{d}_l,d_i}:= \Bigl\{ v\in H^1(G,\mathbb{S}^1); \deg(v, \Gamma_l)=\tilde{d}_l, l=1,\dots,n, \deg(v,\Gamma_0)=\tilde{d}_0-\sum_{i=1}^kd_i\Bigr\}
\end{equation}
and we call \(U_{\tilde{d}_l,d_i}\) a minimizer of the Dirichlet energy in \(\I_{\tilde{d}_l,d_i}\), i.e.,
\begin{equation}\label{eq:min_ref}
U_{\tilde{d}_l,d_i}=\argmin \Bigl\{ \frac12 \int_G |\nabla v|^2; v \in \I_{\tilde{d}_l,d_i}\Bigr\}.
\end{equation}
Again, we will show, in the proof of Theorem \ref{th:main2}, that such a minimizer exists and is unique up to a phase.

\begin{theorem}\label{th:main2}
There exists a  minimizer \(\hat{u}_\rho\) for the problem \eqref{eq:min_w_N_rho} and it is unique modulo to a phase. There exist a subsequence \(\rho_p \to 0\) and a map \(\hat{u}_0\in W^{1,q}(G,\mathbb{S}^1)\) for every \(1\leq q<2\) such that, as \(p \to +\infty\), \(\hat{u}_{\rho_p} \rightarrow \hat{u}_0\) in \(\C^m_{\text{loc}}(G\setminus \{a_1,\dots,a_k\})\) for all \(m \in \mathbb{N}\). The map \(\hat{u}_0\) satisfies
\begin{equation}
\left\{
\begin{array}{rcll}
-\Delta \hat{u}_0&=&|\nabla \hat{u}_0|^2\hat{u}_0 &\text{ in } G \setminus \{a_1,\dots,a_k\},\\
\hat{u}_0\wedge \p_\nu \hat{u}_0&=& 0 &\text{ on } \p G,
\end{array}
\right.
\end{equation}
and we can write
\begin{equation}\label{eq:canonical_map_Neumann}
\hat{u}_0=\prod_{i=1}^k \left( \frac{x-a_i}{|x-a_i|} \right)^{d_i} U_{\deg(\hat{u}_0,\Gamma_l),d_i}e^{i\psi_\N}
\end{equation}
where \(\psi_\N\) is defined in \eqref{eq:phase_Neumann} and \( U_{\deg(\hat{u}_0,\Gamma_l),d_i}\) in \eqref{eq:min_ref}.
Furthermore,
\begin{equation}\label{eq:expansion_Neumann}
W_\N^\rho(\{a_i\},\{d_i\})=\pi \left( \sum_{i=1}^k d_i^2\right) |\log \rho| +W_\N(\{a_i\},\{d_i\})+o(1),
\end{equation}
with
\begin{multline}\label{eq:renormalized_Neumann_explicit}
W_\N(\{a_i\},\{d_i\})= -\pi \sum_{i\neq j} d_id_j \log |a_i-a_j| -\pi \sum_{i=1}^k d_i\hat{R}_0(a_i) 
-\pi \sum_{l=1}^n \sum_{i=1}^k d_i \beta_l\varphi_l(a_i) \\-\frac12\sum_{l=1}^n \sum_{m=1}^n \beta_l\beta_m\int_{\Gamma_l} \p_\nu \varphi_m-\frac12\sum_{l=1}^n \beta_l \int_{\Gamma_l} \p_\nu \hat{R}_0,
\end{multline}

where the functions \(\varphi_l\) satisfy \eqref{eq:def_varphi_l} and the coefficients \(\beta_l=\beta_l(\{a_i\},\{d_i\})\) are real numbers that solve the linear system
\begin{equation}\label{eq:coeff_linear_system_2}
2\pi \deg (\hat{u}_0,\Gamma_l)=\int_{\Gamma_l} \p_\nu \hat{\Phi}_0+\sum_{m=1}^n \beta_m \int_{\Gamma_l} \p_\nu \varphi_m \quad \text{ for } l=1,\dots,n.
\end{equation}
\end{theorem}

Let us briefly indicate the difficulties to pass from simply connected domains to multiply connected ones. For a Dirichlet boundary condition, in \cite{BBH_1994}, the authors proved that the variational problem \eqref{eq:min_w_d_rho} is directly related to a minimization problem whose minimizer solve a linear PDE. Indeed, if \(\O\) is simply connected, and if \(u_\rho\) is a minimizer for \eqref{eq:min_w_d_rho} (which exists by Proposition \ref{prop:existence_Euler_Lagrange}), then \(|\nabla u_\rho|=|\nabla \Phi_\rho|\) where \(\Phi_\rho\) is the harmonic conjugate of the gradient of the phase of \(u_\rho\). The phase is not a well-defined function (this is a multi-valued function) but the gradient of this phase is well-defined and can be expressed by the current \(j(u_\rho)=u_\rho \wedge \nabla u_\rho\) which satisfies \(\dive j(u_\rho)=0\) in \(\O_\rho\). In multiply connected domains, since the Poincar\'e lemma does not necessarily hold, it is not true anymore that the current \(j(u_\rho)\) can be expressed as the perpendicular gradient of a harmonic function and its Hodge decomposition is more complicated. We will show that we can write \( j(u_\rho)=\nabla^\perp \Phi_\rho+ \nabla H_\rho\) where \(\Phi_\rho, H_\rho\) are harmonic functions. Then we study the asymptotic behaviours of these functions as \(\rho\) tends to zero. In \cite{BBH_1994} the main tools to do that were Lemma~\ref{lem:elliptic_estimate_1} and Lemma \ref{lem:elliptic_estimate_2}. We also employ these lemmas for the convergence of \(\Phi_\rho\), however to prove the convergence of \(H_\rho\) we employ a variational argument, cf.\ Lemma \ref{lem:finiteness_renormalized_Dirichlet} and Lemma \ref{lem:energy_of_Phi_rho}, and elliptic estimates. The same difficulties appear in the case of Neumann boundary conditions.

The paper is organized as follows: in section \ref{sec:preliminaries} we recall a generalization of Poincar\'e's lemma giving conditions for a vector field to be written as the gradient of a potential function. We show how it is related to the existence of a harmonic conjugate for a harmonic function and to the existence of a lifting for a \(\mathbb{S}^1\)-valued map. In section \ref{sec:Dirichlet} we study the minimization problem \eqref{eq:min_w_d_rho} and its asymptotics as \(\rho \rightarrow 0\), thus proving Theorem \ref{th:main1}. Section \ref{sec:Neumann} is devoted to the study of a similar minimization problem with Neumann boundary conditions and to the proof of Theorem \ref{th:main2}.  In section \ref{sec:another_approach} we show how the renormalized energies can be obtained by a slightly different approach similar to the point of view in \cite{Ignat_Jerrard_2020}. In the appendix we recall two lemmas presented in \cite[chapter I]{BBH_1994} that are used through this article.\\

\textbf{Acknowledgements.} The second-named author gratefully acknowledges the support of the Paris-Saclay University during this work.
\section{Preliminaries}\label{sec:preliminaries}


We start by stating a generalization of Poincar\'e's lemma which gives condition on which a vector field in \(\R^2\) can be written as the gradient of a function.

\begin{lemma}\label{lem:generalized-Poincare}
Let \(\O\) be a smooth bounded open set in \(\R^2\). Let us call \(\Gamma_1,\dots,\Gamma_N\) the connected components of \(\p \O\). Let \(D\) be a vector field in \(\C^1(\O,\R^2)\cap \C(\overline{\O},\R^2)\) satisfying
\begin{equation}
\left\{
\begin{array}{rcll}
\dive D &=&0 & \text{ in } \O,\\
\int_{\Gamma_i} D \cdot \nu &=&0 & \text{ for } i=1,\dots,N.
\end{array}
\right.
\end{equation}
Then there exists a function \(\phi \in \C^2(\O,\R)\cap \C^1(\overline{\O},\R)\) such that 
\begin{equation}
D=\nabla^\perp \phi:=(-\p_y \phi, \p_x \phi).
\end{equation}
In the same way, if \(D\in \C^1(\O,\R^2)\cap \C(\overline{\O},\R^2)\) is a vector field such that
\begin{equation}
\left\{
\begin{array}{rcll}
\curl D &=&0 & \text{ in } \O,\\
\int_{\Gamma_i} D \cdot \tau &=&0 & \text{ for } i=1,\dots,N.
\end{array}
\right.
\end{equation}
Then there exists a function \(\varphi \in \C^2(\O,\R)\cap\C^1(\overline{\O},\R)\) such that
\begin{equation}
D=\nabla \varphi.
\end{equation}
\end{lemma}

\begin{proof}
This is the same as in Lemma I.1 in \cite{BBH_1994}.
\end{proof}

As a consequence of the previous lemma we have the following criterion to determine when a harmonic function admits a harmonic conjugate.

\begin{lemma}\label{lem:harmonic_conjugate}
 Let \(\O\) be a smooth bounded open domain in \(\R^2\). Let us call \(\Gamma_i, i=1,\dots,N\) the connected components of \(\p \O\). Let \(H\in \C^1(\overline{\O})\) be a harmonic function in \(\O\). Then \(H\) admits a harmonic conjugate, i.e., there exists a harmonic function in \(\O\) denoted by \(H^\perp\) such that 
 \[ \nabla H= \nabla ^\perp H^\perp ,\]
 if and only if 
\begin{equation}
 \int_{\Gamma_i} \p_\nu H =0 \quad \text{ for all } i=1, \dots,N.
\end{equation}
\end{lemma}

\begin{proof}
It suffices to apply the previous lemma with \(D=\nabla H\). Then we find \(H^\perp\) such that \( \nabla H=\nabla^\perp H^\perp\), by observing that \(\curl \nabla H=0\) we obtain that \(H^\perp\) is harmonic. 
\end{proof}

\begin{lemma}\label{lem:lifting}
Let \(\O\) be a smooth bounded open set in \(\R^2\). Let us call \(\Gamma_1,\dots,\Gamma_N\) the connected components of \(\p \O\). Let \( F=(F_1,F_2) \in \C^1(\O,\R^2)\cap \C(\overline{\O},\R^2)\) be such that
\begin{equation}
\curl F=\p_xF_2-\p_yF_1=0,
\end{equation}
\begin{equation}\label{eq:multi_valued}
\int_{\Gamma_i} F\cdot \tau \in 2\pi \mathbb{Z}.
\end{equation}
Then, there exists \(u\in \C^1(\overline{\O},\mathbb{S}^1)\), unique up to a phase, such that
\begin{equation}
j(u):= u\wedge \nabla u= F=(F_1,F_2).
\end{equation}
\end{lemma}
\begin{proof}
We define \(\psi(x)= \int_{\gamma_x} F \cdot \tau\) where \(\gamma_x\) is a path joining a given point \(x_0\) to a point \(x\in \O\). The function \(\psi\) is multi-valued because \(\O\) is possibly multiply connected but, thanks to \eqref{eq:multi_valued}, the different values differ only by an integer multiple of \(2\pi\). Thus \(u= e^{i\psi}\) is well-defined and satisfies that \( j(u)= F\). To prove the uniqueness, we assume that \(u,v\in \C^1(\overline{\O},\mathbb{S}^1)\) are such that \(j(u)=j(v)\). Then, we compute that \(j(u\bar{v})=u\bar{v}\wedge \nabla (u\bar{v})=u\bar{v}\wedge (\bar{v} \nabla u+ u\nabla \bar{v})=j(u)+j(\bar{v})=j(u)-j(v)=0\). But since \(u\bar{v}\) is \(\mathbb{S}^1\)-valued, we have that \( |j(u\bar{v})|=|\nabla (u\bar{v})|\) and it implies that \(\nabla (u \bar{v})=0\) in \(G\). Thus \( u=e^{i\eta} v\) for some \(\eta \in \R\).
\end{proof}

To conclude this section we make the following observation: we define the vector fields \( X_0=(1,1), \{X_l:=\nabla \varphi_l\}_{l=1,\dots,n}\), where the functions \(\varphi_l, l=1,\dots,n\) are defined in \eqref{eq:def_varphi_l}; thanks to Lemma \ref{lem:generalized-Poincare}, \((X_0,X_1,\dots,X_l)\) is a basis of the vector space \[ \{X \in \mathcal{C}^\infty(\overline{G},\R^2); \dive X=\curl X=0, \ X\cdot \tau=0 \text{ on } \p G\}.\] This basis is in duality with a basis of the space of smooth harmonic one-forms in \(G\) with vanishing tangential components. However the basis \( (X_0,\nabla \varphi_1,\dots,\nabla \varphi_l)\) is not orthonormal for the \(L^2\)-inner product since \( \int_G \nabla \varphi_l \cdot \nabla \varphi_m= \int_{\Gamma_l} \p_\nu \varphi_m=\int_{\Gamma_m} \p_\nu \varphi_l\) has no reason to vanish a priori.  
\section{Renormalized energies with Dirichlet boundary conditions}\label{sec:Dirichlet}

Let \(G=\widetilde{G} \setminus \cup_{l=1}^n \overline{\omega}_l\) be a smooth multiply connected bounded domain, with \(\widetilde{G}, \omega_1,\dots, \omega_n\) smooth simply connected bounded domains. We call \(\Gamma_0=\p \widetilde{G}\) the exterior connected component of \(\p G\) and \(\Gamma_1=\p \omega_1,\dots,\Gamma_n=\p \omega_n\) the inner connected components of \(\p G\). These are smooth curves that we orient in an anti-clockwise manner. More precisely \(\nu\) denotes the outward unit normal to \(\p \widetilde{G}\) and the outward unit normal to \(\omega_l\), \(l=1,\dots,n\) and \( (\nu,\tau)\) is always direct, with \(\tau\) a tangent vector to \(\p G\). We take \(g\in \C^1(\p G)\).

 For \(k \in \mathbb{N}^*\), let \(a_1,\dots,a_k\in G\) be \(k\) distinct points in \(G\). For \(\rho>0\) small enough so that \(\bar{B}_\rho(a_i) \cap \bar{B}_\rho(a_j)= \emptyset\) for every \(i \neq j\) and \(\bar{B}_\rho(a_i) \subset \O\) we recall that \(\O_\rho\) is defined by \eqref{eq:Omega_rho}. Our goals in this section is to study the asymptotic behaviour as \(\rho\) goes to \(0\) of the minimization problem \eqref{eq:Omega_rho}, where the class \(\E_{g,\rho}\) is defined in \eqref{eq:E_rho}, and to prove Theorem \ref{th:main1}. In the following, \(\nu\) also denotes the outward unit normal to \(B_\rho(a_i), i=1,\dots,k\).  We start with

\begin{proposition}\label{prop:existence_Euler_Lagrange}
The infimum \eqref{eq:min_w_d_rho} is attained by a map \(u_\rho \in H^1(\O_\rho, \mathbb{S}^1)\) which satisfies the following Euler-Lagrange equation:
\begin{equation}
\left\{
\begin{array}{rcll}
-\Delta u_\rho &=& |\nabla u_\rho|^2u_\rho & \text{ in } \O_\rho, \\
u_\rho&=& g & \text{ on } \p G,\\
u_\rho \wedge \p_\nu u_\rho &=&0  &\text{ on } \p B_\rho(a_i),\ i=1,\dots,k.
\end{array}
\right.
\end{equation}
Furthermore, \(u_\rho \in \C^\infty(\O_\rho,\mathbb{S}^1)\cap \C^1(\overline{\O}_\rho,\mathbb{S}^1)\) and \(u_\rho\) is also smooth up to the boundary of every \(B_\rho(a_i)\) for \(i=1,\dots,k\).
\end{proposition} 

When one uses the direct method of calculus of variations  to prove Proposition \ref{prop:existence_Euler_Lagrange}, the difficulty is that the degree is not continuous with respect to the \(H^{\frac12}\)-weak convergence. However, since we work with \(\mathbb{S}^1\)-valued maps, it is possible to show that, in this particular case, we can recover weak continuity of the degree. This follows for example from a result of White \cite{White_1989}, but we will give a direct proof relying on Lemma \ref{lem:approxiamte_degree} below. We first introduce functions \(V_i\) for \(i=1,\dots,k\) defined by
\begin{equation}\label{def:V_i}
\left\{
\begin{array}{rcll}
-\Delta V_i &=&0 &\text{ in } \O_\rho,\\
V_i&=& 1 &\text{ on } \p \O_\rho \setminus \p B_\rho(a_i), \\
V_i&=& 0 &\text{ on } \p B_\rho(a_i).
\end{array}
\right.
\end{equation}

\begin{lemma}\label{lem:approxiamte_degree}
Let \(u \in H^1(\O_\rho, \mathbb{S}^1),\) then 
\begin{equation}
\deg(u, \p B_\rho(a_i))=\frac{1}{2\pi}\int_{\O_\rho} u \wedge (\p_x V_i \p_y u-\p_yV_i \p_xu).
\end{equation}
Furthermore if \(u,v\in H^1(\O_\rho,\mathbb{S}^1)\) then
\begin{equation}
|\deg(u,\p B_\rho(a_i))- \deg(v,\p B_\rho(a_i))| \leq \frac{2}{\pi} \|V_i\|_{\mathcal{C}^1}\|u-v\|_{L^2}(\|\nabla u\|_{L^2}+\|\nabla v\|_{L^2}).
\end{equation}
\end{lemma}
The proof of this lemma can be found in \cite[section 3]{Berlyand_Rybalko_2010} and \cite[Proposition 1]{Dos_Santos_2009}, we give the details for the comfort of the reader.
\begin{proof}

 For \(i=1,\dots,k\), an integration by parts gives
\begin{align*}
\frac{1}{2\pi} \int_{\O_\rho} u \wedge (\p_x V_i \p_yu-\p_yV_i \p_x u)&=\frac{1}{2\pi} \int_{\Gamma_0}  V_i (u \wedge \p_\tau u) - \frac{1}{2\pi}\int_{\p \O_\rho \setminus \Gamma_0} V_i (u \wedge \p_\tau u) \\
&=\deg(u,\Gamma_0)-\sum_{j=1, j \neq i}^k \deg( u, \p B_\rho(a_i)) -\sum_{l=1}^n \deg(u,\Gamma_l) \\
& = \deg(u, \p B_\rho(a_i)).
\end{align*}
For the last equality we have used that, since \(u \in H^1(\O_\rho,\mathbb{S}^1)\), we have \(\p_xu \cdot \p_yu=0\) almost everywhere and thus by integrating by parts we find
\[ 0= \frac{1}{2\pi}\int_{\O_\rho} \p_x u \wedge \p_y u = \deg(u, \Gamma_0)-\sum_{i=1}^k \deg(u,\p B_\rho(a_i))-\sum_{l=1}^n \deg(u, \Gamma_l).\]
For the second point we observe that since \(V_i\) is locally constant on \(\p \O_\rho\), an integration by parts gives
\[ \int_{\O_\rho} v \wedge (\p_xu \p_yV_i-\p_yu \p_x V_i)=\int_{\O_\rho} u \wedge (\p_xv \p_yV_i-\p_yv \p_x V_i).\]
Hence, by using the first point we find
\begin{multline*}
2\pi|\deg(u,\p B_\rho(a_i))-\deg(v,\p B_\rho(a_i))|  \\
 =\left | \int_{\O_\rho} (u-v) \wedge \left[ (\p_x V_i \p_yu-\p_yV_i \p_xu)+(\p_xV_i\p_yv-\p_yV_i\p_xv) \right] \right| \\
 \leq 4\|u-v\|_{L^2}\|V_i\|_{\mathcal{C}^1}(\|\nabla u\|_{L^2}+\|\nabla v\|_{L^2}).
\end{multline*}
\end{proof}

\begin{proof}(proof of Proposition \ref{prop:existence_Euler_Lagrange}) 
We take a minimizing sequence \((u_n)_n\) for the Dirichlet energy \(E\) in the class \(\E_{g,\rho}\). Since it is bounded in \(H^1\),  we can extract a subsequence weakly converging to some \(u_\rho \in H^1(\O_\rho,\R^2)\). Up to other subsequences, we can assume that \(u_n\) converges strongly to \(u_\rho\) in \(L^2(\O_\rho)\) and \(u_n\) converges almost everywhere to \(u_\rho\). Hence \(u \in H^1(\O_\rho,\mathbb{S}^1)\) and by using Lemma \ref{lem:approxiamte_degree} we find that \( \deg(u_\rho,\p B_\rho(a_i))=\deg(u_n,\p B_\rho(a_i))\) for all \(n\in \mathbb{N}^*\) and \(i=1,\dots,k\). With the weak continuity of the trace operator and the lower semi-continuity of the Dirichlet energy we are able to conclude to the existence. To derive the Euler-Lagrange equations we can make variations of the form \(u_\rho+t\varphi\) for \(t\) small and \(\varphi \in \C^\infty_c(\O_\rho,\R^2)\) and \(u_\rho e^{it \psi}\) for \(t\) small and \(\psi \in \C^\infty( \overline{\O}_\rho,\R)\) with \(\psi\) vanishing on \(\p G\). These variations do preserve the class \(\E_{g,\rho}\). The regularity of \(u_\rho\) follows from the regularity for minimizing harmonic maps due to \cite{Morrey_1948} (see also \cite{Helein_1991}). The regularity up to the boundaries \(\p B_\rho(a_i)\) can be proved as in \cite[Lemma 4.4]{Berlyand_Mironescu_unpublished}.
\end{proof}

In the rest of the paper we will make an intensive use of the \textit{current} of a function.
\begin{definition}
If \( U \subset \R^2\) is a  bounded open set, for \(u \in L^p\cap W^{1,p'}(U,\mathbb{C})\), with \(1 \leq p,p'\leq +\infty\) and \(1/p+1/p'=1\), we define the \textit{current} associated to \(u\) by
\begin{equation}
j(u):=u \wedge \nabla u= (u\wedge \p_xu, u \wedge \p_y u).
\end{equation}
\end{definition}

\begin{lemma}\label{lem:equation_current}
If \(u_\rho\) is a solution of the minimization problem \eqref{eq:min_w_d_rho} given by Proposition~\ref{prop:existence_Euler_Lagrange} then its current \(j(u_\rho)\) satisfies
\begin{equation}
\left\{
\begin{array}{rcll}
\dive j({u_\rho})=0  & \text{ and } & \curl j(u_\rho)=0 \text{ in } \O_\rho,\\
j({u_\rho}) \cdot \nu =0  \text{ on } \p B_\rho(a_i), \ i=1,\dots, k,  & \text{ and } & j({u_\rho}) \cdot \tau = g \wedge \p_\tau g  \text{ on } \p G.
\end{array}
\right.
\end{equation}
\end{lemma}

\begin{proof}
We compute
\begin{align*}
\dive j(u_\rho)&= \p_x (u_\rho \wedge \p_xu_\rho)+\p_y (u_\rho \wedge \p_y u_\rho) =u_\rho \wedge \Delta u_\rho = 0.
\end{align*}
In the same way 
\[ \p_x (u_\rho \wedge \p_y u_\rho)-\p_y (u_\rho \wedge \p_x u_\rho)=2 \p_x u_\rho \wedge \p_y u_\rho =0.\]
The information for \(j(u_\rho)\) on the boundary comes from the information on \(u_\rho\) on the boundary.
\end{proof}
We now use the generalized Poincar\'e Lemma \ref{lem:generalized-Poincare} to derive a Hodge decomposition of the current \(j(u_\rho)\). First we prove

\begin{proposition}\label{prop:exist_Phi_rho}
There exists a function \(\Phi_\rho \in H^1(\O_\rho,\R)\), satisfying
\begin{equation}\label{eq:phi_rho}
\left\{
\begin{array}{rcll}
\Delta \Phi_\rho &=&0 &\text{ in } \O_\rho,\\
\Phi_\rho&=& \text{cst.} & \text{ on } \p B_\rho(a_i),\ i=1,\dots,k,\\
\int_{\p B_\rho(a_i)} \p_\nu \Phi_\rho &=& 2\pi d_i & \text{ for }i=1,\dots, k, \\
\p_\nu \Phi_\rho &=& g \wedge \p_\tau g &\text{ on } \p G, \\
\int_{\p G} \Phi_\rho &=&0.
\end{array}
\right.
\end{equation}
We also have \(\Phi_\rho \in \C^\infty(\O_\rho,\R)\cap \C^1(\overline{\O}_\rho,\R)\) and is smooth up to the boundaries \(\p B_\rho(a_i), i=1,\dots,k\).

Moreover, there exist a unique \(v_\rho \in H^1(\O_\rho,\mathbb{S}^1)\) and a unique \(\theta_{l,\rho}\in [-\pi,\pi[\), \(l=1,\dots,n\), such that \( j({v_\rho})=\nabla^\perp \Phi_\rho\) in \(\O_\rho\), \(v_\rho=g\) on \( \Gamma_0 \subset \p G\) and \(v_\rho=e^{-i\theta_{l,\rho}}g\) on \(\Gamma_l\), \(l=1,\dots,n\).
\end{proposition}

\begin{proof}
The existence follows from the fact that a solution to \eqref{eq:phi_rho} is a minimizer of
\[ \frac12 \int_{\O_\rho} |\nabla \phi|^2+2\pi \sum_{i=1}^k d_i \phi_{|\p B_\rho(a_i)} -\int_{\p G} \phi \left( g \wedge \p_\tau g \right) \]
in the space \[ V_\rho:=\{\phi\in H^1(\O_\rho,\R); \varphi=\text{cst.}=\phi_{|\p B_\rho(a_i)} \text{ on each } \p B_\rho(a_i), \ i=1,\dots,k, \int_{\p G} \varphi=0\}.\] The uniqueness follows because the functional to be minimized is strictly convex in \(V_\rho\). The smoothness of \(\Phi_\rho\) follows from the regularity for harmonic functions.
Since the vector field \( (-\p_y \Phi_\rho, \p_x \Phi_\rho)\) satisfies the assumption of Lemma \ref{lem:lifting}, we can find a function \( v_\rho \in H^1(\O_\rho,\mathbb{S}^1)\) satisfying \( j({v_\rho})=\nabla^\perp \Phi_\rho\) in \(\O_\rho\), furthermore \(v_\rho \) is unique up to a phase.
Since \(v_\rho\) satisfies
\[ v_\rho \wedge \p_\tau v_\rho= \p_\nu\Phi_\rho =g \wedge \p_\tau g \text{ on } \p G \]
we can choose an appropriate phase to prescribe \( v_\rho =g \) on \( \Gamma_0\), and then we have \( v_\rho=e^{-i \theta_{l,\rho}}g \) on each \(\Gamma_l\), \(l=1,\dots,n\) for some \( \theta_{l,\rho}=\theta_{l,\rho}(g,\{a_i\},\{d_i\}) \in [-\pi,\pi[\).
\end{proof}

\begin{proposition}\label{prop:Phi_rho}
Let \(u_\rho\) be a solution of \eqref{eq:min_w_d_rho} and let \(\Phi_\rho\) be a solution of \eqref{eq:phi_rho} then there exists \(H_\rho\in \C^\infty(\overline{\O}_\rho,\R)\) such that
\begin{equation}\label{eq:Hodge_of_urho}
j({u_\rho})= \nabla^\perp \Phi_\rho +\nabla H_\rho, \quad u_\rho=v_\rho e^{i H_\rho},
\end{equation}
\begin{equation}\label{eq:H_rho}
\left\{
\begin{array}{rcll}
\Delta H_\rho &=&0 &\text{ in } \O_\rho, \\
\p_\nu H_\rho &=&0 &\text{ on } \p B_\rho(a_i),\ i=1,\dots,k,\\
H_\rho &=& \alpha_{l,\rho}(g,\{a_i\},\{d_i\}) &\text{ on } \Gamma_l,\ l=1\dots,n, \\
H_\rho&=&0  &\text{ on } \Gamma_0,
\end{array}
\right.
\end{equation}
where \(\alpha_{l,\rho}(g,\{a_i\},\{d_i\})=\theta_{l,\rho}+2\pi\mathbb{Z}\), \(l=1,\dots,n\), with \(\theta_{l,\rho}\) defined in Proposition \ref{prop:exist_Phi_rho}. In particular, we have uniqueness of a solution of \eqref{eq:min_w_d_rho}.
\end{proposition}

\begin{proof}
This follows from Lemma \ref{lem:generalized-Poincare} since \(j({u_\rho})-\nabla^\perp \Phi_\rho\) verifies that \(\curl \left(j(u_\rho)- \nabla^\perp \Phi_\rho \right)=0\), \( \left(j(u_\rho)\cdot \tau- \p_\nu \Phi_\rho\right) =g\wedge \p_\tau g- g \wedge \p_\tau g=0\) on \(\p G\) and \(\int_{\p B_\rho(a_i)} \left(j(u_\rho) \cdot \tau- \p_\nu \Phi_\rho\right)=2\pi(d_i-d_i)=0\) for \(i=1,\dots,k\). Note that \(H_\rho\) is defined up to a constant and that is why we can impose \(H_\rho=0\) on \(\Gamma_0\). Now we have that
\begin{align*}
 j(v_\rho e^{iH_\rho})&=  v_\rho e^{iH_\rho} \wedge \nabla (v_\rho e^{iH_\rho})
 = v_\rho e^{i H_\rho} \wedge  \left( \nabla v_\rho e^{iH_\rho}+i v_\rho \nabla H_\rho e^{i H_\rho} \right) \\
 &=v_\rho \wedge \nabla v_\rho +\nabla H_\rho =\nabla^\perp \Phi_\rho +\nabla H_\rho.
\end{align*}
From Lemma \ref{lem:lifting}, this means that \(u_\rho=v_\rho e^{i H_\rho}\) up to a  phase, but since \(H_\rho=0\) on \( \Gamma_0\) and \(v_\rho= g \) on \(\Gamma_0\) we have \( u_\rho =v_\rho e^{i H_\rho}.\) Hence, since from Proposition \ref{prop:exist_Phi_rho} \(v_\rho=e^{-i\theta_{l,\rho}}g\) on \(\Gamma_l\), \(l=1,\dots,n\) we find that \(\alpha_{l,\rho}=\theta_{l,\rho}+2\pi \mathbb{Z}\). This means also that \(u_\rho\) is uniquely determined.
\end{proof}

We recall that \(\Phi_0\in \C^\infty( G \setminus \{a_1,\dots,a_k\})\) is the solution of \eqref{eq:Phi_0}.

\begin{lemma}\label{lem:existence_v_0}
 There exists a unique  \(v_0 \in \C^\infty ( G \setminus \{a_1,\dots,a_k\}, \mathbb{S}^1)\) such that
 \begin{equation}
 \left\{
 \begin{array}{rcll}
  j(v_0)&=&v_0 \wedge \nabla v_0= \nabla^\perp \Phi_0 &\text{ in } G \setminus \{a_1,\dots,a_k\}, \\
 v_0&=&g & \text{ on } \Gamma_0 \subset G, \\
 v_0&=&e^{-i \theta_l} g & \text{ on } \Gamma_l, \ l=1,\dots,n,
\end{array}
\right.
 \end{equation}
 for some \(\theta_l=\theta_l(g,\{a_i\},\{d_i\}) \in [-\pi,\pi[\) for \(l=1,\dots,n\).
\end{lemma}

\begin{proof}
The proof of Lemma \ref{lem:lifting} can be adapted to this context to find this \(v_0\). Note that \(\p_\nu \Phi_0=g\wedge \p_\tau g=v_0\wedge \p_\tau v_0\) on \(\p G\) implies that on every connected component \(\Gamma_0,\Gamma_1,\dots,\Gamma_{n}\) of \(\p G\) we can write \(v_0=e^{-i \theta_l}g\), furthermore we can choose \(\theta_0=0\).
\end{proof}

\begin{proposition}\label{prop:conv_phi_rho}
Let \(\Phi_\rho\) be the solution to \eqref{eq:phi_rho} and \(\Phi_0\) the solution to \eqref{eq:Phi_0}, then for every \(m\in \mathbb{N}\) and every compact set \(K \subset \overline{G} \setminus \{a_1,\dots,a_k\}\) there exists \(C_{m,K}\) such that 
\begin{equation}\label{eq:estimates_derivatives}
\|\Phi_\rho-\Phi_0\|_{\C^m(K)} \leq C_{m,K} \rho .
\end{equation}
\end{proposition}

\begin{proof}
We apply Lemma \ref{lem:elliptic_estimate_1} to \(v_\rho= \Phi_0-\Phi_\rho\) which satisfies, \(\Delta v_\rho=0\) in \(\O_\rho\), \(\p_\nu v_\rho=0\) on \(\p G\), \(\int_{\p B_\rho(a_i)} \p_\nu v_\rho=0\) for \(i=1,\dots,k\).  Hence, since \(\Phi_\rho\) is constant on \(\p B_\rho(a_i)\) we find
\begin{align*}
\sup_{\O_\rho} v_\rho- \inf_{\O_\rho} v_\rho \leq \sum_{i=1}^k \left(\sup_{\p  B_\rho(a_i)} v_\rho- \inf_{\p B_\rho(a_i)} v_\rho\right) \leq  \sum_{i=1}^k \left(\sup_{\p B_\rho(a_i)} \Phi_0- \inf_{\p B_\rho(a_i)} \Phi_0 \right)\leq C\rho.
\end{align*}
Since \(\int_{\p G} (\Phi_\rho-\Phi_0)=0\), there exists a point \(x \in \p G\) such that \((\Phi_\rho-\Phi_0)(x)=0\), thus we find that \( \|\Phi_\rho-\Phi_0\|_{L^\infty(\O_\rho)} \leq C\rho\). By elliptic estimates, see e.g.\ \cite[Theorem 2.10]{Gilbarg_Trudinger}, we obtain \eqref{eq:estimates_derivatives}.
\end{proof}

We introduce 
\begin{equation}
\tilde{\E}_{g,\rho}:= \Bigl\{ v \in H^1(\O_\rho,\mathbb{S}^1); v=g \text{ on } \p G; v=\left(\frac{x-a_i}{|x-a_i|}\right)^{d_i}, \ i=1,\dots,k\Bigr\}
\end{equation}
\begin{equation}\label{eq:minu_tilde_rho}
\tilde{W}_g^\rho(\{a_i\},\{d_i\}):= \inf_{v\in \tilde{\E}_{g,\rho}} \frac12 \int_{\O_\rho} |\nabla v|^2.
\end{equation}

\begin{lemma}\label{lem:finiteness_renormalized_Dirichlet}
The map \(\rho \mapsto W_g^\rho(\{a_i\},\{d_i\})-\pi \left( \sum_{i=1}^k d_i^2\right) |\log \rho|\) is non-increasing, whereas the map \( \rho \mapsto \tilde{W}_g^\rho(\{a_i\},\{d_i\})-\pi \left( \sum_{i=1}^k d_i^2\right) |\log \rho|\) is non-decreasing. Furthermore\\ \(W_g^\rho(\{a_i\},\{d_i\}) \leq\  \tilde{W}_g^\rho(\{a_i\},\{d_i\})\) and we have that 
\begin{equation}
W_g(\{a_i\},\{d_i\}):=\lim_{\rho \to 0} \left( W_g^\rho(\{a_i\},\{d_i\})-\pi \left( \sum_{i=1}^k d_i^2\right) |\log \rho| \right) \text{ exists and is finite}.
\end{equation}
\end{lemma}

This lemma follows from Proposition 2.10 and Lemma 2.11 in \cite{Monteil_Rodiac_VanSchatfingen_2020a}. We reproduce the proof for the comfort of the reader.
\begin{proof}
Let \(0<\rho<\sigma\) with \(\sigma\) small enough so that the balls \( \overline{B}_\sigma(a_i)\) are disjoints and included in \(G\). We can write that
\begin{equation}\nonumber
\int_{G \setminus \cup_{i=1}^k \overline{B}_\rho(a_i)} |\nabla u_\rho|^2 = \int_{G \setminus \cup_{i=1}^k \overline{B}_\sigma(a_i)} |\nabla u_\rho|^2 +\int_{\cup_{i=1}^k \left(B_\sigma(a_i) \setminus \overline{B}_\rho(a_i)\right)} |\nabla u_\rho|^2.
\end{equation}
Now, by using polar coordinates centred at \(a_i\) and the Cauchy-Schwarz inequality, we have
\begin{align*}
\int_{B_\sigma(a_i)\setminus \overline{B}_\rho(a_i)} |\nabla u_\rho|^2 & \geq \int_\rho^\sigma \int_0^{2\pi} \frac{|\p_{\theta_i}u_\rho|^2}{r_i} \dif \theta_i  \dif r_i  \geq \frac{1}{2\pi} \int_\rho^{\sigma} \frac{1}{r_i} \left| \int_0^{2\pi} (u_\rho \wedge \p_{\theta_i} u_\rho) \dif \theta_i \right|^2 \dif r_i \\
& \geq 2\pi d_i^2 \log\frac{\sigma}{\rho}.
\end{align*}
Thus we find that
\begin{equation}\label{eq:intermediaire}
W_{g}^{\rho}-\pi \left(\sum_{i=1}^k d_i^2\right)\log \frac{1}{\rho} \geq W_{g}^{\sigma}-\pi\left( \sum_{i=1}^k d_i^2 \right) \log \frac{1}{\sigma}
\end{equation}
which proves the first assertion (here and in the rest of the proof, for simplicity, we do not write the dependence of the singularities \( \{a_i\},\{d_i\}\)). For the second assertion, if \( 0<\rho<\sigma\) and if \(\tilde{u}_\sigma\) is a minimizer for the problem \(\tilde{W}_{g}^{\sigma}\), then the map \[v(x)=\begin{cases} \tilde{u}_\sigma(x) & \text{ if } x\in \O_\sigma \\
\left(\frac{x-a_i}{|x-a_i|}\right)^{d_i} & \text{ if } x \in B_\sigma(a_i) \setminus \overline{B}_\rho(a_i) \end{cases}\] is a comparison map from the minimization problem \( \tilde{W}_{g}^{\rho}\). Thus
\begin{align*}
 \tilde{W}_{g}^{\rho}  &\leq \frac12 \int_{\O_\rho} |\nabla v|^2 = \tilde{W}_{g}^{\sigma}+\pi \left(\sum_{i=1}^k d_i^2\right) \log \frac{\sigma}{\rho}.
\end{align*}
This proves that \(\rho \mapsto \tilde{W}_{g}^{\rho}\) is non-decreasing. We can easily see that \( W_{g}^{\rho} \leq \tilde{W}_{g}^{\rho}\) for every \(\rho\). Hence both quantities admit a limit when \(\rho\) goes to zero and their limits are finite.
\end{proof}

\begin{proposition}\label{prop:conv_u_rho_1}
Let \(u_\rho\) be the solution to the minimization problem \eqref{eq:min_w_d_rho}. Then there exist \(u_0\in H^1_{\text{loc}}(\overline{G} \setminus \{a_1,\dots,a_k\}, \mathbb{S}^1)\), \(H_0 \in H^1_{\text{loc}}(\overline{G} \setminus \{a_1,\dots,a_k\})\) and a sequence \(\rho_p \rightarrow 0\) such that
\begin{equation}\nonumber
u_{\rho_p} \rightharpoonup u_0 \text{ in } H^1_{\text{loc}}(\overline{G} \setminus \{a_1,\dots,a_k\}), \quad H_\rho \rightharpoonup H_0  \text{ in } H^1_{\text{loc}}(\overline{G} \setminus \{a_1,\dots,a_k\}).
\end{equation}
\end{proposition}

\begin{proof}
The proof follows the idea of \cite[Proposition 8.1]{Monteil_Rodiac_VanSchatfingen_2020a}. By Lemma \ref{lem:finiteness_renormalized_Dirichlet}, for \(0<\rho<\sigma\) we have
\[ \int_{G \setminus \cup_{i=1}^k \bar{B}_\sigma(a_i)} \frac{|\nabla u_\rho|^2}{2}\leq W_g^\rho(\{a_i\},\{d_i\})-\pi\left(\sum_{i=1}^k d_i^2\right) \log \frac{\sigma}{\rho}.\]
By using Lemma \ref{lem:finiteness_renormalized_Dirichlet} again we arrive at
\begin{equation}\label{eq:boundedness_condition}
\int_{G \setminus \cup_{i=1}^k \bar{B}_\sigma(a_i)} \frac{|\nabla u_\rho|^2}{2} \leq W_g(\{a_i\},\{d_i\})-\pi \left( \sum_{i=1}^k d_i^2\right) \log \sigma.
\end{equation}
Thanks to the boundedness condition \eqref{eq:boundedness_condition} we can use a diagonal argument to find a subsequence \(\rho_p \to 0\) and a map \(u_0 \in H^1_{\text{loc}}(\overline{G} \setminus \{a_1,\dots,a_k\},\mathbb{S}^1)\) such that
\(u_{\rho_p} \rightharpoonup u_0\) in \(H^1_{\text{loc}}(\overline{G} \setminus \{a_1,\dots,a_k\})\).
Now since we know from Proposition \ref{prop:conv_phi_rho} that \( \Phi_\rho \rightarrow \Phi_0\) in \( \C^m_{\text{loc}}(\overline{G} \setminus \{a_1,\dots,a_k\}) \) we find that \(\nabla H_{\rho_p}=j(u_{\rho_p})-\nabla^\perp \Phi_{\rho_p}\) converges weakly in \(L^2_{\text{loc}} (\overline{G} \setminus \{ a_1,\dots,a_k\})\). From the Poincar\'e inequality, which is valid here since \(H_\rho=0\) on \( \Gamma_0\), we infer that there exists \(H_0\in H^1_{\text{loc}}(\overline{G} \setminus \{a_1,\dots,a_k\})\) such that \( H_{\rho_p} \rightharpoonup H_0\) in \( H^1_{\text{loc}}(\overline{G} \setminus \{a_1,\dots,a_k\})\).
\end{proof}

In particular, from the previous proposition and the  weak continuity of the trace operator, there exist a subsequence \(\rho_p \to 0\) and \(\alpha_l=\alpha_l(g,\{a_i\},\{d_i\})\) such that
\begin{equation}\label{eq:conv_alpha_l_rho}
\alpha_{l,\rho} \rightarrow \alpha_l, \quad \text{ for } l=1,\dots,n.
\end{equation}

\begin{lemma}\label{lem:energy_of_Phi_rho}
Let \(\Phi_\rho\) be a solution to \eqref{eq:phi_rho}. Then,  we have
\begin{multline}\nonumber
\lim_{\rho \to 0} \left( \frac12 \int_{\O_\rho} |\nabla \Phi_\rho|^2-\pi \left( \sum_{i=1}^kd_i^2\right)|\log\rho| \right)= -\pi \sum_{i\neq j} d_id_j \log |a_i-a_j| \\+\frac12 \int_{\p G} \Phi_0(g \wedge \p_\tau g) -\pi \sum_{i=1}^k d_iR_0(a_i) <+\infty.
\end{multline}
\end{lemma}

\begin{proof}
An integration by parts gives
\begin{equation*}
\int_{\O_\rho} |\nabla \Phi_\rho|^2=\int_{\Gamma_0}\p_\nu \Phi_\rho \Phi_\rho
-\sum_{l=1}^n \int_{\Gamma_l} \p_\nu \Phi_\rho \Phi_\rho  -\sum_{i=1}^k \int_{\p B_\rho(a_i)} \p_\nu \Phi_\rho \Phi_\rho .
\end{equation*}
Now we use \eqref{eq:phi_rho} and more particularly we use  \(\p_\nu \Phi_\rho=g \wedge \p_\tau g\) on \(\Gamma_l, l=0,1,\dots,n\), \(\Phi_\rho=cst.\) on \(\p B_\rho(a_i)\), \(i=1,\dots,k\) and \(\int_{\p B_\rho(a_i)} \p_\nu \Phi_\rho=2\pi d_i\) to obtain
\begin{align*}
\int_{\O_\rho} |\nabla \Phi_\rho|^2 &= \int_{\Gamma_0} (g\wedge \p_\tau g) \Phi_\rho -\sum_{l=1}^n \int_{\Gamma_l} (g\wedge \p_\tau g) \Phi_\rho - \sum_{i=1}^k 2\pi d_i \Phi_\rho(\p B_\rho(a_i)) \\
&= \int_{\Gamma_0} (g \wedge \p_\tau g) \Phi_0-\sum_{l=1}^n \int_{\Gamma_l} (g\wedge \p_\tau g)\Phi_0-\sum_{i=1}^k 2\pi d_i \Phi_0(x_i) +O(\rho)
\end{align*}
where \(x_i\) is a point in \(\p B_\rho(a_i)\). Since \(R_0(x)=\Phi_0(x)-\sum_{i=1}^k d_i \log |x-a_i|\) we can write
\begin{align*}
\int_{\O_\rho}|\nabla \Phi_\rho|^2&= \int_{\Gamma_0} (g\wedge \p_\tau g)\Phi_0-\sum_{l=1}^n \int_{\Gamma_l} (g\wedge \p_\tau g)\Phi_0 +\sum_{i=1}^k 2\pi d_i^2 |\log \rho|\\
& -\sum_{i\neq j} 2\pi d_i d_j \log |a_i-a_j|-\sum_{i=1}^k 2\pi d_i R_0(a_i)+O(\rho).
\end{align*}
This yields the result.
\end{proof}

\begin{proposition}\label{prop:conv_H_rho}
Let \(H_\rho\) be the solution to \eqref{eq:H_rho}, then up to a subsequence \(\rho_p \to 0\), we can find \( \alpha_l=\alpha_l(g,\{a_i\},\{d_i\})\) for \(l=1,\dots,n\) and \(H_0 \in \C^\infty(\overline{G}\setminus \{a_1,\dots,a_k\})\) such that
\begin{equation}\label{eq:conv_H_rho}
H_{\rho_p} \rightarrow H_0 \text{ in } \C^m_{\text{loc}}(\overline{G}\setminus \{a_1,\dots,a_k\}) \text{ for every } m \in \mathbb{N}
\end{equation}
with \(H_0\) satisfying
\begin{equation}\label{eq:H_0}
\left\{
\begin{array}{rcll}
\Delta H_0&=&0 & \text{ in } G, \\
H_0&=&0 & \text{ on } \Gamma_0, \\
H_0&=& \alpha_l & \text{ on }  \Gamma_l, \ l=1,\dots,n.
\end{array}
\right.
\end{equation}
\end{proposition}

\begin{proof}
We already know from Proposition \ref{prop:conv_u_rho_1} that there exist \(\alpha_l\), \(l=1,\dots,n\) and \(H_0 \in H^1_{\text{loc}}(\overline{G}\setminus \{a_1,\dots,a_k\})\) such that, up to a subsequence not labelled, \(H_\rho \rightharpoonup H_0\) in \(H^1_{\text{loc}}(\overline{G}\setminus \{a_1,\dots,a_k\})\). It remains to show that \(H_0\) satisfies \eqref{eq:H_0}. First by elliptic estimates, cf.\ e.g.\ \cite[Theorem 5.21]{Giaquinta_Martinazzi_2012}, we have find that \eqref{eq:conv_H_rho} holds and \(\Delta H_0=0\) in \(G \setminus \{a_1,\dots,a_k\}\), \(H_0=0\) on \(\Gamma_0\) and \(H_0=\alpha_l\) on \(\Gamma_l, l=1,\dots,n\). But we can use \eqref{eq:Hodge_of_urho}, \( |\nabla u_\rho|^2=|j(u_\rho)|^2\) and an integration by parts to write
\begin{align*}
\int_{\O_\rho} |\nabla u_\rho|^2& = \int_{\O_\rho} |\nabla \Phi_\rho|^2+\int_{\O_\rho} |\nabla H_\rho|^2+ 2\int_{\O_\rho} \nabla^\perp \Phi_\rho \cdot \nabla H_\rho \\
&= \int_{\O_\rho} |\nabla \Phi_\rho|^2+\int_{\O_\rho} |\nabla H_\rho|^2+ 2\sum_{l=1}^n \int_{\Gamma_l} \alpha_{l,\rho} \p_\tau \Phi_\rho.
\end{align*}
By Proposition \ref{prop:conv_phi_rho} and since \(\alpha_{l,\rho} \rightarrow \alpha_l\) as \( \rho \rightarrow 0\)  we find that \( \int_{\Gamma_l} \alpha_{l,\rho} \p_\tau \Phi_\rho \rightarrow \int_{\Gamma_l} \alpha_l \p_\tau \Phi_0\) as \(\rho \rightarrow 0\). From Lemma \ref{lem:finiteness_renormalized_Dirichlet} and Lemma \ref{lem:energy_of_Phi_rho} we obtain that
\[ \sup_{\rho>0} \int_{\O_\rho} |\nabla H_\rho|^2 < +\infty.\]
By lower semi-continuity of the Dirichlet energy, for every \(\sigma>0\) we have
\begin{align*}
\int_{\O_\sigma} |\nabla H_0|^2 \leq \liminf_{\rho \to 0} \int_{\O_\sigma} |\nabla H_\rho|^2 \leq \limsup_{\rho \to 0} \int_{\O_\sigma} |\nabla H_\rho|^2.
\end{align*}
But if \(\rho<\sigma\) then \( \int_{\O_\sigma} |\nabla H_\rho|^2 \leq \int_{\O_\rho} |\nabla H_\rho|^2\) and hence we arrive at
\[ \sup_{\sigma>0} \int_{\O_\sigma} |\nabla H_0|^2<+\infty.\]
By monotone convergence, it implies that \(\nabla H_0 \in L^2(G)\), and by the Poincaré inequality we find that \(H_0 \in L^2(G)\). Then it can be show that the singularities \(a_1,\dots,a_k\) are removable\footnote{To prove this we can take a cut-off function \(\eta\) such that \(\eta\equiv 1\) in \(B_\e(a_i)\) and \(\eta\equiv 0\) in \(B_{2\e}(a_i)^c\), \(i=1,\dots,k\), then we write that \(\int_G \nabla H_0\nabla \psi =\int_G \nabla H_0 \nabla [\psi (1-\eta) +\psi \eta]\). By using that \( \|\nabla \eta\|_{L^\infty}\leq C/\e\), and that  \( \|\nabla H_0\|_{L^2(B_\e(a_i))}\to 0\) as \(\e \to 0\), we arrive at the result.} for \(H_0\) and thus \(\Delta H_0=0\) in \(G\).
\end{proof}

\begin{proposition}\label{prop:conv_u_rho_22}
Let \(u_\rho\) be the solution to the minimization problem \eqref{eq:min_w_d_rho}, then there exists a sequence \(\rho_p \to 0\) such that
\begin{equation}
u_{\rho_p} \rightarrow u_0 \text{ in } \C^m_{\text{loc}}(G\setminus \{a_1,\dots,a_k\}) \text{ for every } m\in \mathbb{N}
\end{equation}
with \(u_0\in \C^\infty(G \setminus \{a_1,\dots,a_k\},\mathbb{S}^1)\) satisfying
\begin{equation}
\left\{
\begin{array}{rcll}
-\Delta u_0 &=& |\nabla u_0|^2u_0 &\text{ in } G \setminus \{a_1,\dots,a_k\},\\
u_0&=&g & \text{ on } \p G.
\end{array}
\right.
\end{equation}
Furthermore we have that 
\begin{equation}\label{eq:Hodge_of_u0}
j(u_0)=\nabla^\perp \Phi_0+\nabla H_0
\end{equation}
and \(u_0=v_0e^{iH_0}\). In particular \(u_0\) satisfies \( \dive j(u_0)=0, \curl j(u_0)=2\pi \sum_{i=1}^k d_i \delta_{a_i}\) in \(G\) and \(j(u_0)\cdot \tau=g\wedge \p_\tau g
\) on \(\p G\).
\end{proposition}
\begin{proof}
This result follows from the convergence of \(\Phi_\rho\) in Proposition \ref{prop:conv_phi_rho} and \(H_\rho\) in Proposition~\ref{prop:conv_H_rho}.
\end{proof}

We are now in position to obtain Theorem \ref{th:main1}

\begin{proof}(proof of Theorem \ref{th:main1})
 From Lemma \ref{lem:finiteness_renormalized_Dirichlet} we know that the limit of \(W_g^\rho(\{a_i\},\{d_i\})-\pi \left( \sum_{i=1}^k d_i^2\right) |\log \rho|\) as \(\rho \to 0\) exists and is finite. To compute this limit we can use a special subsequence \(\rho_p\to 0\) such that Proposition \ref{prop:conv_u_rho_22} and Proposition \ref{prop:conv_H_rho} hold. For simplicity of notation, in the rest of the proof we let \(\rho=\rho_p\). Let \(u_\rho\) be the solution to the minimization problem \eqref{eq:min_w_d_rho}. We use that \(|\nabla u_\rho|^2=|j(u_\rho)|^2\) along with \eqref{eq:Hodge_of_urho}  and an integration by parts to obtain
\begin{align*}
\int_{\O_\rho} |\nabla u_\rho|^2&= \int_{\O_\rho} |\nabla \Phi_\rho|^2+\int_{\O_\rho}|\nabla H_\rho|^2+2\int_{\O_\rho} \nabla^\perp \Phi_\rho \cdot \nabla H_\rho \\
&= \int_{\O_\rho} |\nabla \Phi_\rho|^2-\sum_{l=1}^n \int_{\Gamma_l} \alpha_{l,\rho} \p_\nu H_\rho +2 \sum_{l=1}^n \int_{\Gamma_l} \alpha_{l,\rho} \p_\tau \Phi_\rho.
\end{align*}
We have used the boundary condition for \(H_\rho \) in \eqref{eq:H_rho}. Now we use Lemma \ref{lem:energy_of_Phi_rho} and Proposition~\ref{prop:conv_H_rho} to obtain 
\begin{multline*}
\int_{\O_\rho} |\nabla u_\rho|^2=\int_{\Gamma_0}\Phi_0(g\wedge \p_\tau g)-\sum_{l=1}^n \int_{\Gamma_l} \Phi_0(g \wedge \p_\tau g)+2\pi\sum_{i=1}^k d_i^2|\log \rho| \\
-2\pi\sum_{i\neq j} d_i d_j \log |a_i-a_j|-2\pi\sum_{i=1}^kd_iR_0(a_i)-\sum_{l=1}^n\int_{\Gamma_l} \alpha_l \p_\nu H_0 \\+2\sum_{l=1}^n \int_{\Gamma_l} \alpha_{l} \p_\tau \Phi_0 +o_\rho(1).
\end{multline*}
We can integrate by parts once more and find that 
\begin{equation}\label{eq:energy_H_0}
\sum_{l=1}^n\int_{\Gamma_l} \alpha_l \p_\nu H_0=\int_G |\nabla H_0|^2.
\end{equation}
We decompose \(H_0=\sum_{l=1}^n \alpha_l \varphi_l(x)\)  where the functions \(\varphi_l\) are defined by \eqref{eq:def_varphi_l} to find \eqref{eq:renormalized_Dirichlet}.  We also observe that
\begin{equation}
\int_G|\nabla H_0|^2=\sum_{l=1}^n \int_{\Gamma_l} \alpha_l \p_\nu H_0 = \sum_{l=1}^n \sum_{m=1}^n \alpha_l \alpha_m \int_{\Gamma_l} \p_\nu \varphi_m.
\end{equation}
 Now we describe the coefficients \(\alpha_l=\alpha_l(g,\{a_i\},\{d_i\})\). To see that \(\alpha_l=\theta_l+2\pi\mathbb{Z}\) we recall from Lemma \ref{lem:existence_v_0} that \(v_0=e^{-i\theta_l}g\) on \(\Gamma_l\), \(l=1,\dots,n\). Then, we observe that \(j(u_0\overline{v}_0)=\nabla H_0\) and thus we can conclude that \(u_0=v_0e^{iH_0}e^{i\eta}\) for some constant \(\eta\in \R\). Since \(u_0=v_0\) on \(\Gamma_0\), and \(H_0=0\) on \(\Gamma_0\) we obtain that \(\eta \in 2\pi\mathbb{Z}\). Besides, on each \(\Gamma_l\), \(l=1,\dots,n\) we obtain that \(g=ge^{-i\theta_l+i\alpha_l+i\eta}\) which implies \(\alpha_l=\theta_l+2\pi\mathbb{Z}\).  Next we take the inner product of \eqref{eq:Hodge_of_u0} with \( \nabla \varphi_l\)  to find
\begin{equation}
\int_G j(u_0)\cdot \nabla \varphi_l =\sum_{m=1}^n \alpha_m \int_{\Gamma_m} \p_\nu \varphi_l \quad \text{ for } l=1,\dots,n.
\end{equation}
By integrating by parts the left-hand side we arrive at \eqref{eq:alpha_l_sol_linear_system}.
It remains to show that \(u_0\) is given by \eqref{eq:canocnical_harmonic_map}. First, by using the same arguments as in Proposition \ref{prop:existence_Euler_Lagrange} we can see that there exists a minimizer \(U_{g,d_i}\) of the Dirichlet energy in \(\I_{g,d_i}\) where this class is defined in \eqref{eq:class_Dirichlet}. To prove the uniqueness up to a multiplication by a constant. We write the Euler-Lagrange equations for \(U_{g,d_i}\) and we use Lemma \ref{lem:generalized-Poincare} to prove that \(j(U_{g,d_i})=\nabla^\perp \Phi_U\) for some function \(\Phi_U\). We use again the Euler-Lagrange equations on \(U_{g,d_i}\) to obtain that \(\Phi_U\) satisfies
\begin{equation}
\left\{
\begin{array}{rcll}
\Delta \Phi_U &=&0 & \text{ in } G, \\
\Phi_U&=&\text{cst.} & \text{ on } \Gamma_l, \ l=0,1,\dots,n,\\
 \int_{\Gamma_l}\p_\nu \Phi_U&=& \deg (g,\Gamma_l) & \ l=1,\dots,n, \\
 \int_{\Gamma_0} \p_\nu \Phi_U& =&\deg(g,\Gamma_0)-\sum_{i=1}^k d_i.
\end{array}
\right.
\end{equation}
As in Proposition \ref{prop:Phi_rho} we obtain that \(\Phi_U\) is uniquely determined up to a constant, since it is a minimizer 
of \( F(\varphi)= \frac12 \int_G |\nabla \varphi|^2+2\pi \sum_{l=1}^n \varphi \deg(g,\Gamma_l)+2\pi \varphi\left( \deg(g,
\Gamma_0)-\sum_{i=1}^kd_i\right)\) in the space \[ \{ \varphi \in H^1(G,\R); \varphi=cst. \text{ on } \Gamma_l, \ l=0,1,\dots,n 
\}. \] This minimizer is unique up to a constant by a convexity argument. We then use Lemma~\ref{lem:lifting}.  Now we call \(V:= u_0\overline{U}_{g,d_i} \prod_{i=1}^k\left(\frac{\overline{x-a_i}}{|x-a_i|} \right)^{d_i}\) and we compute that 
\begin{equation}
j\left(V \right)=j(u_0)-\sum_{i=1}^k d_i\left( \frac{(x-a_i)^\perp}{|x-a_i|^2} \right)-\nabla^\perp\Phi_U.
\end{equation}
Thus we can check that \( \curl j\left(V\right)=0\) in \(G\), \(\int_{\Gamma_l} V\cdot \tau=0\) for \(l=0,1,\dots,n\). By applying Lemma~\ref{lem:generalized-Poincare} we find \(\psi_g \in \C^1(G)\) such that \(j(V)=\nabla \psi_g\). Therefore, by using Lemma \ref{lem:lifting} we can write \(V=e^{i\psi_g}\) which yields \eqref{eq:canocnical_harmonic_map}. We can check that \( \dive j(V)=0\) in \(G\)  and hence we find that \(\Delta \psi_g=0\) in \(G\).
\end{proof}

\section{Renormalized energies with Neumann boundary conditions}\label{sec:Neumann}

In this section we fix \(k \in \mathbb{N}\), \( d_i=1,\dots,k,\) and we consider \(\E_{\N,\rho}\) given by \eqref{eq:class_min_Neumann} and \(W_\N^\rho\) given by \eqref{eq:min_w_N_rho}.

\begin{proposition}\label{prop:existence_Neumann}
The infimum \(W_\N^\rho\) in \eqref{eq:min_w_N_rho} is attained. Let \(\hat{u}_\rho\) be a minimizer  for \eqref{eq:min_w_N_rho} then \(\hat{u}_\rho \) satisfies
\begin{equation}
\left\{
\begin{array}{rcll}
-\Delta \hat{u}_\rho &=& |\nabla \hat{u}_\rho|^2 \hat{u}_\rho & \text{ in } \O_\rho, \\
\hat{u}_\rho \wedge \p_\nu \hat{u}_\rho &=&0 & \text{ on } \p \O_\rho, \\
|\hat{u}_\rho|&=&1 & \text{ on } \O_\rho.
\end{array}
\right.
\end{equation}
We also have that \(\hat{u}_\rho \in \C^\infty(\overline{\O}_\rho,\mathbb{S}^1)\), \(\int_{\p B_\rho(a_i)}\hat{u}_\rho \wedge \p_\tau \hat{u}_\rho=2\pi d_i,\) for \(i=1,\dots,k\), \(\int_{\Gamma_l} \hat{u}_\rho \wedge \p_\tau \hat{u}_\rho=: 2\pi \tilde{d}_l \in 2\pi \mathbb{Z}\), for \(l=0,1,\dots,n\) and \(\sum_{i=1}^k d_i+\sum_{l=1}^n \tilde{d}_l=\tilde{d}_0\).
\end{proposition}

\begin{proof}
The proof follows the same lines as in Proposition \ref{prop:existence_Euler_Lagrange}.
\end{proof}
As in the previous section, we introduce the current associated with \(\hat{u}_\rho\) defined by \(j(\hat{u}_\rho):= \hat{u}_\rho \wedge \nabla \hat{u}_\rho\).

\begin{lemma}\label{prop:properties_current_Neumann}
Let \(\hat{u}_\rho\) be a minimizer for \eqref{eq:min_w_N_rho} then the current \(j(\hat{u}_\rho)\) satisfies
\begin{equation}
\left\{
\begin{array}{rcll}
\dive j(\hat{u}_\rho) =0 & \text{ and }& \curl j(\hat{u}_\rho) = 0 &  \text{ in } \O_\rho,\\
j(\hat{u}_\rho)\cdot \nu =0 & & &\text{ on } \p \O_\rho.
\end{array}
\right.
\end{equation}
Furthermore we have \(\int_{\p B_\rho(a_i)} j(\hat{u}_\rho) \cdot \tau =2\pi d_i\) for \(i=1,\dots,k\) and \( \int_{\Gamma_l} j(\hat{u}_\rho) \cdot \tau =2\pi \tilde{d}_l \in 2\pi \mathbb{Z}\), for \(l=0,1,\dots,n\).
\end{lemma}
The proof is similar to the one of Lemma \ref{lem:equation_current}.
Thanks to the previous lemma we can apply the generalized Poincar\'e lemma to obtain:
\begin{proposition}\label{prop:phi_rho_Neumann}
There exists a unique \(\hat{\Phi}_\rho \in \C^\infty(\overline{\O}_\rho,\R)\) such that
\[ j(\hat{u}_\rho)=\hat{u}_\rho \wedge \nabla \hat{u}_\rho = \nabla^\perp \hat{\Phi}_\rho,\]
and
\begin{equation}
\left\{
\begin{array}{rcll}
\Delta \hat{\Phi}_\rho&=&0 & \text{ in } \O_\rho,\\
\hat{\Phi}_\rho&=&0 & \text{ on } \Gamma_0, \\
\hat{\Phi}_\rho&=& \tilde{\beta}_{i,\rho} & \text{ on } \p B_\rho(a_i), \ i=1,\dots,k,\\
\hat{\Phi}_\rho&=& \beta_{i,\rho} & \text{ on } \Gamma_l, \ l=1,\dots,n,
\end{array}
\right.
\end{equation}
with \(\beta_{i,\rho}, \tilde{\beta}_{i,\rho}\) being real constants. Furthermore we have
\begin{align*}
 \int_{\p B_\rho(a_i)}\p_\nu \hat{\Phi}_\rho&=2\pi d_i, \text{ for } i=1,\dots,k, \ \int_{\Gamma_l} \p_\nu \hat{\Phi}_\rho = 2\pi\tilde{d}_l \in 2\pi \mathbb{Z}, \text{ for } l=0,1,\dots,n \\
   \sum_{i=1}^k d_i +\sum_{l=1}^n \tilde{d}_l&=\tilde{d}_0.
\end{align*}
\end{proposition}

\begin{proof}
The existence comes from Lemma \ref{lem:generalized-Poincare} and the properties of the current \(j(\hat{u}_\rho)\) gathered in Lemma \ref{prop:properties_current_Neumann}. The uniqueness follows since \(\hat{\Phi}_\rho\) in Proposition \ref{prop:phi_rho_Neumann} is the minimizer  of 
\[ F(\varphi) =\frac12 \int_{\O_\rho} |\nabla \varphi|^2+2\pi \sum_{i=1}^k d_i \varphi_{|\p B_\rho(a_i)}+
2\pi \sum_{l=1}^n \tilde{d}_l\varphi_{|\Gamma_l} \]
in the class
\begin{multline}\nonumber \hat{V}_\rho= \{ \varphi \in H^1(\O_\rho,\R); \varphi=0 \text{ on } \Gamma_0, \varphi=cst.=\varphi_{|\p B_\rho(a_i)} \text{ on } \p B_\rho(a_i), i=1,\dots,k \\
 \ \varphi=cst.=\varphi_{|\Gamma_l} \text{ on } \Gamma_l, l=1,\dots,n\}.
\end{multline}
\end{proof}

\begin{corollary}
The minimizer \(\hat{u}_\rho\) for \eqref{eq:min_w_N_rho} is unique up to a multiplication by a complex constant of modulus \(1\).
\end{corollary}
\begin{proof}
It comes from the uniqueness of \(\hat{\Phi}_\rho\) and Lemma \ref{lem:lifting}.
\end{proof}

We  now define 
\begin{equation}\nonumber
\tilde{\E}_{\N,\rho}:= \Bigl\{v \in H^1(\O_\rho,\mathbb{S}^1); \tr_{|\p B_\rho(a_i)}=\left( \frac{x-a_i}{|x-a_i|} \right)^{d_i} ,\ i=1,\dots,k  \Bigr\}
\end{equation}
\text{and}
\begin{equation*}
 \tilde{W}_\N^\rho(\{a_i\},\{d_i\}):=\inf_{v \in \tilde{\E}_{\N,\rho}} \frac12 \int_{\O_\rho} |\nabla v|^2.    
\end{equation*}

Thanks to this other variational problem we can show, as in Lemma \ref{lem:finiteness_renormalized_Dirichlet},
\begin{lemma}\label{lem:finiteness_renormalized_Neumann}
The map \(\rho \mapsto W_\N^\rho(\{a_i\},\{d_i\})-\pi \left( \sum_{i=1}^k d_i^2\right) |\log \rho|\) is non-increasing, and the map \( \rho \mapsto \tilde{W}_\N^\rho(\{a_i\},\{d_i\})-\pi \left( \sum_{i=1}^k d_i^2\right) |\log \rho|\) is non-decreasing. Furthermore \\
 \( W_\N^\rho(\{a_i\},\{d_i\}) \leq \tilde{W}_\N^\rho(\{a_i\},\{d_i\}) \) and we have that 
\begin{equation}\nonumber
W_\N(\{a_i\},\{d_i\}):=\lim_{\rho \to 0} \left( W_\N^\rho(\{a_i\},\{d_i\})-\pi \left( \sum_{i=1}^k d_i^2\right) |\log \rho| \right) \text{ exists and is finite}.
\end{equation}
\end{lemma}

\begin{proposition}\label{prop:conv_u_rho_et_phi_rho}
Let \(\hat{u}_\rho\) be a solution to the minimization problem \eqref{eq:min_w_N_rho} and let \(\hat{\Phi}_\rho\) be given by Proposition \ref{prop:phi_rho_Neumann}. Then there exist \(\hat{u}_0\in H^1_{\text{loc}}(\overline{G}\setminus \{a_1,\dots,a_k\},\mathbb{S}^1)\), \( \tilde{\Phi}_0 \in H^1_{\text{loc}}(\overline{G} \setminus \{a_1,\dots,a_k\})\) and \(\rho_p \to 0\) such that
\begin{equation}
\hat{u}_{\rho_p} \rightharpoonup \hat{u}_0 \text{ in } H^1_{\text{loc}}(\overline{G} \setminus \{a_1,\dots,a_k\}), \hat{\Phi}_{\rho_p} \rightharpoonup \tilde{\Phi}_0 \text{ in } H^1_{\text{loc}}(\overline{G} \setminus \{a_1,\dots,a_k\}).
\end{equation}
In particular, there exist \(\beta_l=\beta_l(\{a_i\},\{d_i\})\) such that \({\hat{\Phi}_\rho}{|_{\Gamma_l}}=\beta_{l,\rho} \rightarrow \beta_l:={\tilde{\Phi}_0}{|_{\Gamma_l}}\) for \(l=1,\dots,n\).
\end{proposition}

The proof of this proposition is similar to the proof of Proposition \ref{prop:conv_u_rho_1}. This relies on a diagonal argument and Lemma \ref{lem:finiteness_renormalized_Neumann}.

 Note that \({\tilde{\Phi}_0}{|_{\Gamma_l}}\) is a constant since it is the limit of constant real numbers. Now we define \(\hat{\Phi}_0\) to be the solution of 
\begin{equation}\label{phi_0_Neumann}
\left\{
\begin{array}{rcll}
\Delta \hat{\Phi}_0 &=& 2\pi \sum_{i=1}^k d_i \delta_{a_i} & \text{ in } G, \\
\hat{\Phi}_0 &=& 0 & \text{ on } \Gamma_0,\\
\hat{\Phi}_0&=& \beta_l & \text{ on } \Gamma_l, \ l=1,\dots,n. 
\end{array}
\right.
\end{equation}

\begin{proposition}\label{prop:conv_phi_rho_chapeau}
The equality \(\tilde{\Phi}_0=\hat{\Phi}_0\) holds and, up to a subsequence \(\rho_p \to 0\), \( \hat{\Phi}_{\rho_p} \rightarrow \hat{\Phi}_0\) in \(\C^m_{\text{loc}}(\overline{G}\setminus \{a_1,\dots,a_k\})\) for all \(m \in \mathbb{N}\). 
\end{proposition}

\begin{proof}
We take \(\rho_p\to 0\) as in Proposition \ref{prop:conv_u_rho_22}, for notational simplicity we denote \(\rho=\rho_p\). We apply Lemma \ref{lem:elliptic_estimate_2} to \(v_\rho:= \hat{\Phi}_0-\hat{\Phi}_\rho\) which satisfies \(\Delta v_\rho=0\) in \(\O_\rho\), \( \int_{\p B_\rho(a_i)} \p_\nu v=0\) for \(i=1,\dots,k\). Since \(\hat{\Phi}_\rho\) is constant on \(\p B_\rho(a_i)\) we find
\begin{align*}
\sup_{\O_\rho} v_\rho- \inf_{\O_\rho} v_\rho & \leq \sum_{i=1}^k \left( \sup_{\p B_\rho(a_i)} v_\rho -\inf_{\p B_\rho(a_i)}v_\rho \right) +\sup_{\p G} v_\rho -\inf_{\p G} v_\rho \\
& \leq \sum_{i=1}^k \left( \sup_{\p B_\rho(a_i)} \hat{\Phi}_0-\inf_{\p B_\rho(a_i)} \hat{\Phi}_0 \right)+\sup_{\p G} v_\rho -\inf_{\p G} v_\rho \\
& \leq C \rho + \sup_{\p G} v_\rho -\inf_{\p G} v_\rho.
\end{align*}
But since \( \hat{\Phi}_\rho\) and \(\hat{\Phi}_0\) are constants on each connected components of \(\p G\) and \( (\hat{\Phi}_0-\hat{\Phi}_\rho)_{| \Gamma_l} \rightarrow 0\) for every \(l=0,1,\dots,n\) we find that \(\sup_{\p G} v_\rho -\inf_{\p G} v_\rho \rightarrow 0\) as \(\rho \to 0\). Now we use that \(v_\rho=\hat{\Phi}_0-\hat{\Phi}_\rho=0\) on \(\Gamma_0\) to obtain that \( \|\hat{\Phi}_0-\hat{\Phi}_\rho\|_{L^\infty(\O_\rho)}=o_\rho(1)\). The conclusion follows from elliptic estimates, cf.\  \cite[Theorem 2.10]{Gilbarg_Trudinger}.
\end{proof}

\begin{proposition}\label{prop:u_rho_chapeau_forte}
Let \(\hat{u}_\rho\) be a solution to the minimization problem \eqref{eq:min_w_N_rho}. Then, up to a subsequence \(\rho_p \to 0\), we have that \( \hat{u}_{\rho_p} \rightarrow \hat{u}_0\) in \( \mathcal{C}^m_{\text{loc}}(G \setminus \{a_1,\dots,a_k\})\) for every \(m \in \mathbb{N}\), where \(\hat{u}_0\) is given by Proposition \ref{prop:conv_u_rho_et_phi_rho}.
\end{proposition}
\begin{proof}
This follows from Proposition \ref{prop:conv_phi_rho_chapeau} and Proposition \ref{prop:phi_rho_Neumann}.
\end{proof}

We call \( \hat{G}_0\) the solution to \eqref{eq:Green_Dirichlet} and \( \hat{R}_0\) the regular part of this Green function defined by \eqref{eq:regular_part_Dirichlet}.
 Then we can write
\begin{equation}
\hat{\Phi}_0(x)=\sum_{i=1}^k d_i \log |x-a_i|+\sum_{l=1}^n \beta_l \varphi_l(x) +\hat{R}_0(x),
\end{equation}
where the functions \(\varphi_l\) are defined in \eqref{eq:def_varphi_l}. We are ready to prove Theorem \ref{th:main2}.

\begin{proof}(proof of Theorem \ref{th:main2})
From Lemma \ref{lem:finiteness_renormalized_Neumann} we know that the limit of \(W_g^\rho(\{a_i\},\{d_i\})-\pi \left( \sum_{i=1}^k d_i^2\right) |\log \rho|\) as \(\rho \to 0\) exists and is finite. To compute this limit we can use a special subsequence \(\rho_p\to 0\) such that Proposition \ref{prop:conv_phi_rho_chapeau} and Proposition \ref{prop:u_rho_chapeau_forte} hold. For simplicity of notation, in the rest of the proof we let \(\rho=\rho_p\). We compute
\begin{align*}
\int_{\O_\rho} |\nabla \hat{u}_\rho|^2 &=\int_{\O_\rho} |\nabla \hat{\Phi}_\rho|^2= \int_{\Gamma_0} \hat{\Phi}_\rho\p_\nu \hat{\Phi}_\rho-\sum_{i=1}^k \int_{\p B_\rho(a_i)} \hat{\Phi}_\rho \p_\nu \hat{\Phi}_\rho-\sum_{l=1}^n \int_{\Gamma_l} \hat{\Phi}_\rho \p_\nu \hat{\Phi}_\rho .
\end{align*}
Since \(\hat{\Phi}_\rho=0\) on \(\Gamma_0\) and \(\hat{\Phi}_\rho\) is constant on \(\p B_\rho(a_i) \) we can write, for \(x_i \in \p B_\rho(a_i)\):
\begin{align*}
\int_{\O_\rho} |\nabla \hat{u}_\rho|^2 &=-2\pi \sum_{i=1}^k \hat{\Phi}_\rho(x_i) d_i -\sum_{l=1}^n \int_{\Gamma_l} \hat{\Phi}_\rho \p_\nu \hat{\Phi}_\rho \\
&=-2\pi \sum_{i=1}^k \hat{\Phi}_0(x_i)d_i-\sum_{l=1}^n \int_{\Gamma_l } \hat{\Phi}_0 \p_\nu \hat{\Phi}_0+o_\rho(1) \\
&= 2\pi \sum_{i=1}^k d_i^2 |\log \rho| -2\pi \sum_{i \neq j} d_id_j \log |a_i-a_j| - 2\pi \sum_{i=1}^k d_i\hat{R}_0(a_i) \\
& \quad -2\pi \sum_{l=1}^n \sum_{i=1}^k d_i \beta_l \varphi_l(a_i)-\sum_{l=1}^n  \int_{\Gamma_l} \hat{\Phi}_0 \p_\nu \hat{\Phi}_0  +o_\rho(1).
\end{align*}
We now use that \(\hat{\Phi}_0=\beta_l\) on \(\Gamma_l\) and we observe that
\[ \int_{\Gamma_l} \p_\nu \hat{\Phi}_0 =\sum_{m=1}^n \int_{\Gamma_l} \beta_m \p_\nu \varphi_m+\int_{\Gamma_l} \p_\nu \hat{R}_0 \quad \text{ for } l=1,\dots,n.\]
We conclude that
\begin{multline}
\frac12 \int_{\O_\rho} |\nabla \hat{u}_\rho|^2=\pi \left( \sum_{i=1}^k d_i^2 \right)|\log \rho| -\pi \sum_{i\neq j} d_i d_j \log |a_i-a_j| -\pi \sum_{i=1}^k d_i \hat{R}_0(a_i) \\
-\pi \sum_{l=1}^n \sum_{i=1}^k d_i \beta_l \varphi_l(a_i)-\frac12\sum_{l=1}^n \sum_{m=1}^n \beta_l \beta_m \int_{\Gamma_l} \p_\nu \varphi_m-\frac12\sum_{l=1}^n \beta_l \int_{\Gamma_l} \p_\nu \hat{R}_0+o_\rho(1).
\end{multline}
This yields \eqref{eq:expansion_Neumann} with the expression of \(W_\N(\{a_i\},\{d_i\})\) given by \eqref{eq:renormalized_Neumann_explicit}.  We now turn to the task of expressing the coefficients \(\beta_l=\beta_l(\{a_i\},\{d_i\})\). Recall that we have
\[ j(\hat{u}_0)=\nabla^\perp \hat{\Phi}_0= \nabla^\perp \left( \hat{G}_0+\sum_{l=1}^n \beta_l \varphi_l \right).  \]
We take the inner product with \(\nabla^\perp \varphi_l\) and integrate by parts to find that
\begin{equation*}
\int_G j(\hat{u}_0)\cdot \nabla^\perp \varphi_l =-2\pi\sum_{i=1}^k d_i \varphi_l(a_i)+  \int_{\Gamma_l} \p_\nu \hat{G}_0 +\sum_{m=1}^n\beta_m \int_{\Gamma_l} \p_\nu \varphi_m. 
\end{equation*}
On the other hand, we have
\begin{align*}
\int_G j(\hat{u}_0)\cdot \nabla^\perp \varphi_l&=-\int_G \curl j(\hat{u}_0) \varphi_l +\int_{\p G} \varphi_l [j(\hat{u}_0)\cdot \tau] \\
&=-2\pi \sum_{i=1}^k d_i \varphi_l(a_i)+\int_{\Gamma_l} j(\hat{u}_0)\cdot \tau \\
&=-2\pi \sum_{i=1}^k d_i \varphi_l(a_i)+2\pi \tilde{d}_l.
\end{align*}

Thus we find that the coefficients \(\beta_l\) solve the linear system
\[ 2\pi \tilde{d}_l=\int_{\Gamma_l} \p_\nu \hat{G}_0 +\sum_{m=1}^n \beta_m \int_{\Gamma_l} \p_\nu \varphi_m=2\pi\sum_{i=1}^k d_i\varphi_l(a_i)+\sum_{m=1}^n \beta_m \int_{\Gamma_l} \p_\nu \varphi_m.\]
The last equality being obtained by multiplying \(\Delta \hat{G}_0\) by \( \varphi_l\) and integrating by parts.

It remains to show that \(\hat{u}_0\) is given by \eqref{eq:canonical_map_Neumann}. We call  \(\tilde{d}_l:=\deg(\hat{u}_0,\Gamma_l)\), \(l=0,1,\dots,n\).We have that \(\sum_{l=1}^n \tilde{d}_l+\sum_{i=1}^k d_i=\tilde{d}_0\). Again, by using the same arguments as in Proposition~\ref{prop:existence_Euler_Lagrange} we can see that there exists a minimizer \(U_{\tilde{d}_l,d_i}\) of the Dirichlet energy in \(\I_{\tilde{d}_l,d_i}\) where this class is defined in \eqref{eq:class_I_Neumann}. We write the Euler-Lagrange equations for \(U_{\tilde{d}_l,d_i}\) and we use Lemma \ref{lem:generalized-Poincare} to prove that \(j(U_{\tilde{d}_l,d_i})=\nabla^\perp \hat{\Phi}_U\). We use again the Euler-Lagrange equations on \(U_{\tilde{d}_l,d_i}\) to obtain that \(\hat{\Phi}_U\) satisfies
\begin{equation}
\left\{
\begin{array}{rcll}
\Delta \hat{\Phi}_U &=&0 & \text{in } G, \\
\hat{\Phi}_U&=&\text{cst.} & \text{on each connected component of } \p G,\\
\int_{\Gamma_l} \p_\nu \hat{\Phi}_U&=&2\pi \tilde{d}_l  & \text{for } l=1,\dots,n.
\end{array}
\right.
\end{equation}
Thus \(\hat{\Phi}_U\) is uniquely determined up to a constant, since it is a minimizer of \( F(\varphi)= \frac12 \int_G |\nabla \varphi|^2-2\pi\sum_{l=1}^n \tilde{d}_l \varphi_{|{\Gamma_l}}\) in the space \[ \{ \varphi \in H^1(G,\R); \varphi=cst. \text{ on each connected component of } \p G \}.\] This minimizer is unique up to a constant by a convexity argument. By Lemma \ref{lem:lifting}, the uniqueness of \(U_{\tilde{d}_l,d_i}\) holds, up to a constant.
We then set \[ \hat{V}:=\hat{u}_0 \prod_{i=1}^k \left( \frac{\overline{x-a_i}}{|x-a_i|^2}\right)^{d_i}\overline{U}_{\tilde{d}_l,d_i}.\] As in the proof of Theorem \ref{th:main1} we can show that 
\begin{equation}
j(\hat{V})=j(\hat{u}_0)-\sum_{i=1}^k d_i\left(\frac{(x-a_i)^\perp}{|x-a_i|^2}\right)-\nabla^\perp \hat{\Phi}_U.
\end{equation}
Thus we can show that 
\begin{align*}
\dive j(\hat{V}) &= 0 \text{ in } G, \quad \curl j(\hat{V})=0  \text{ in } G, \\
j(\hat{V})\cdot \nu &=-\sum_{i=1}^k d_{i}\frac{(x-a_i)^\perp}{|x-a_i|^2}\cdot \nu  \text{ on } \p G, & \quad \int_{\Gamma_l} j(\hat{V})\cdot \tau = 0  \text{ for } l=0,1,\dots,n.
\end{align*}

By Lemma \ref{lem:generalized-Poincare} we can find \(\psi\) such that \( j(\hat{V})=\nabla \psi\). We can also see that \(\psi\) satisfies \(\Delta \psi=0\) in \(G\) and \( \p_\nu \psi=-\sum_{i=1}^k d_{i}\frac{(x-a_i)^\perp}{|x-a_i|^2}\cdot \nu\) on \(\p G\). By uniqueness, up to a constant of such boundary value problem we can assume that \(\psi=\psi_\N\) where \(\psi_\N\) is defined in \eqref{eq:phase_Neumann}.
By using Lemma~\ref{lem:lifting}, this prove that, up to a multiplication by a constant \(\hat{V}=e^{i\psi_\N}\) and this yields \eqref{eq:canonical_map_Neumann}.
\end{proof}

We conclude this section by two remarks:

\begin{remark}
We were not able to decide if the optimal degree configuration for \(\hat{u}_0\) is \\ \(\deg(\hat{u}_0,\Gamma_l)=0\) for \(l=1,\dots,n\) and \(\deg (\hat{u}_0,\Gamma_0)=\sum_{i=1}^k d_i\). This is the situation assumed in \cite{delPino_Kowalczyk_Musso_2006} where the authors can suppose that since their goal is to find a critical point of the Ginzburg-Landau energy with homogeneous Neumann boundary condition.
\end{remark}

\begin{remark}
Except for the Dirichlet and the Neumann boundary problems, a third boundary condition is sometimes considered in the Ginzburg-Landau literature. This is the so-called semi-stiff problem where one prescribes \(|u|=1\) on \(\p G\) with fixed degrees on each components of \(\p G\), cf.\ e.g.,\ \cite{Berlyand_Mironescu_2006a,Berlyand_Mironescu_unpublished,Berlyand_Rybalko_2010,Dos_Santos_2009,Mironescu_2013,Berlyand_Mironescu_Rybalko_Sandier_2014,Rodiac_Sandier_2014,DosSantos_Rodiac_2016, Rodiac_2019}. In this case, minimizers of the Ginzburg-Landau energy do not always exist. However, a natural renormalized energy that we can associate to this problem is the same as in the homogeneous Neumann boundary condition but with fixed degrees, i.e., the degrees of a limiting maps are fixed and we do not optimize the energy on these degrees. Hence the expression of the renormalized energy is given by \eqref{eq:renormalized_Neumann_explicit} where the coefficients \(\beta_l\) are determined by the same system \eqref{eq:coeff_linear_system_2} but with \(\tilde{d}_l=\deg(u_0,\Gamma_l)\) fixed in advance for \(l=1,\dots,n\). Also, the limiting locations of vortices of the Ginzburg-Landau energy are minimizers of this renormalized energy on all \(\overline{G}\) and these vortices can escape through the boundary. When it happens, it is shown in \cite{Berlyand_Rybalko_Yio_2012} that vortices tend to escape through points of maximal curvature of the boundary.
\end{remark}

\section{ Another approach to renormalized energies}\label{sec:another_approach}

In this section, we propose an alternative approach to define the renormalized energies. We first define particular singular harmonic maps with prescribed singularities and then associate a renormalized energy to these maps by taking the Dirichlet energy outside of small balls around the singularities minus the diverging part of this energy. The renormalized energy derived in the previous section is then the infimum of the renormalized energies among all singular harmonic maps with prescribed singularities. This is the approach of \cite{Ignat_Jerrard_2020}. We note that when  \(G\) is simply connected, our singular harmonic map with prescribed singularities is unique (modulo a phase for the Neumann problem) and corresponds to the canonical harmonic map defined in  \cite{BBH_1994}. Due to the multiply connectedness of the domain, uniqueness does not hold in our case.

\subsection{Dirichlet boundary conditions}

Let \(g\in \C^1(\p G,\mathbb{S}^1)\). Let \(u_*\in W^{1,1}(G,\mathbb{S}^1)\), we say that \(u_*\) is a singular harmonic map with prescribed singularities \\ \( (a_1,d_1),\dots,(a_k,d_k)\) if \(u_*\) satisfies
\begin{equation}\label{eq:canonical}
\left\{
\begin{array}{rcll}
\curl j(u_*) = 2\pi \sum_{i=1}^kd_i \delta_{a_i}  \text{ in } G, \quad \dive j(u_*)= 0  \text{ in } G, \\
j(u_*)\cdot \tau = g \wedge \p_\tau g  \text{ on } \p G.
\end{array}
\right.
\end{equation}

\begin{proposition}\label{prop:alternative_approach_1}
Let \(u_*\in W^{1,1}(G,\mathbb{S}^1)\) satisfying \eqref{eq:canonical}, then we can write
\begin{equation}\label{eq:decompo_u_*}
j(u_*)= \nabla^\perp \Phi_0+\nabla H_*
\end{equation}
where \(\Phi_0\) is the solution to \eqref{eq:Phi_0} and \(H_*\) a solution to 
\begin{equation}\label{eq:H_*}
\left\{
\begin{array}{rcll}
\Delta H_*&=&0 &\text{ in } G,\\
H_*&=&0 & \text{ on } \Gamma_0, \\
H_*&=& \alpha_l^* & \text{ on } \Gamma_l, \ l=1,\dots,n.
\end{array}
\right.
\end{equation}
The coefficients \(\alpha_l^*\) are given as the solution to the linear system
\begin{equation}\label{eq:coeff_linear_system}
\sum_{l=1}^n \alpha^*_l \int_{\Gamma_l} \p_\nu \varphi_m = \int_{\Gamma_m} j(u_*)\cdot \nu \quad \text{ for } m=1,\dots,n,
\end{equation}
where the functions \(\varphi_l\) are defined in \eqref{eq:def_varphi_l}.
Moreover, there exist \( \theta_l=\theta_l(g,\{a_i\},\{d_i\})\in [-\pi,\pi[\), \(l=1,\dots,n\) such that for every \(u_*\) satisfying \eqref{eq:canonical}, the associated coefficients \(\alpha_l^*=\alpha_l^*(g,\{a_i\},\{d_i\})\) defined by \eqref{eq:coeff_linear_system} verify \(\alpha_l^*=\theta_l+2\pi \mathbb{Z}\).

\end{proposition}

\begin{proof}
We observe that \( \curl ( j(u_*)-\nabla^\perp \Phi_0)=0\) in \(G\) and \((j(u_*)-\nabla^\perp \Phi_0)\cdot \tau=0\) on \(\p G\). We can apply Lemma \ref{lem:generalized-Poincare} to find \(H^*\) such that \eqref{eq:decompo_u_*} holds. By using that \(\dive j(u_*)=0\) in \(G\) and \((j(u_*)-\nabla^\perp \Phi_0)\cdot \tau=0\) on \(\p G\), we find that there exist constant coefficients \(\alpha_l^*\) such that \eqref{eq:H_*} holds. To express these coefficients, we multiply \eqref{eq:decompo_u_*} by \(\nabla \varphi_m\) and  integrate by parts for \(m=1,\dots,n\) with \(\varphi_m\) defined in \eqref{eq:def_varphi_l}. To see that the coefficients \(\alpha_l^*\) satisfy the quantization property, we recall from Lemma \ref{lem:existence_v_0}, that there exists \(v_0 \in \C^\infty(G \setminus \{a_1,\dots,a_k\},\mathbb{S}^1)\) such that \(j(v_0)=\nabla^\perp \Phi_0\) and \(v_0=g\) on \(\Gamma_0\). We also have that \(v_0=e^{-i\theta_l} g\) for some \(\theta_l\in [-\pi,\pi[\) for \(l=1,\dots,n\) because we have that \( \p_\nu \Phi_0=g\wedge \p_\tau g\) on \(\p G\). But we can check that \(j(v_0e^{iH_*})=j(u_*)\). Indeed
\begin{align*}
v_0e^{iH_*}\wedge \nabla (v_0e^{i H_*})&= v_0e^{iH_*} \wedge (\nabla v_0+iv_0\nabla H_*)e^{i H_*} \\
&= v_0\wedge \nabla v_0+\nabla H_*= \nabla ^\perp \Phi_0+\nabla H_*.
\end{align*}
Since we also have that \( u_*=v_0e^{iH_*}\) on \(\Gamma_0\) we necessarily find that \(u_*=v_0e^{iH_*}\) in \(G\). This implies that \(\alpha_l^*=\theta_l+2\pi \mathbb{Z}\) on each \(\Gamma_l, l=1,\dots,n.\)

\end{proof}

We now show that we can define the renormalized energy of such a map \(u_*\).

\begin{proposition}
For \(u_*\in W^{1,1}(G,\mathbb{S}^1)\) satisfying \eqref{eq:canonical}
\begin{equation}
W_g(u_*):= \lim_{\rho \to 0} \frac{1}{2}\int_{\O_\rho} |\nabla u_*|^2-\pi\left( \sum_{i=1}^k d_i^2 \right)\log \frac{1}{\rho} 
\end{equation}
exists, is finite and is equal to \eqref{eq:renormalized_Dirichlet} where the coefficients \(\alpha_l\) are replaced by \(\alpha_l^*\). Furthermore
\begin{equation}
W_g(\{a_i\},\{d_i\})=W_g(u_0)=\min\{ W_g(u_*); u_* \in W^{1,1}(G,\mathbb{S}^1) \text{ satisfies } \eqref{eq:canonical}\}.
\end{equation}
\end{proposition}
\begin{proof}
The same kind of computations as in the proof of Theorem \ref{th:main1} show that \(W_g(u_*)<+\infty\) for every \(u_*\in W^{1,1}(G,\mathbb{S}^1)\) satisfying \eqref{eq:canonical} and give an expression of this quantity similar to \eqref{eq:renormalized_Dirichlet}.
Now, since \(u_0\) satisfies \eqref{eq:canonical} we have that \[W_g(u_0) \geq \inf \{ W(u_*); u_* \in W^{1,1}(G,\mathbb{S}^1) \text{ satisfies } \eqref{eq:canonical} \}.\] If there exists \(u_*\in W^{1,1}(G,\mathbb{S}^1)\) satisfying \eqref{eq:canonical} such that \(W_g(u_*)<W_g(u_0)\), then we can set \(u_{*,\rho}:={u_*}_{| \O_\rho}\). For \(\rho\) small enough we have \(W_g^\rho(u_{*,\rho})-\pi \left(\sum_{i=1}^k d_i^2 \right)|\log \rho | <W_g(u_0)\). But, this implies that, for \(\rho\) small enough,
\begin{equation}\nonumber
W_g^\rho(\{a_i\},\{d_i\})-\pi \left(\sum_{i=1}^k d_i^2 \right)|\log \rho | \leq W_g^\rho(u_{*,\rho})-\pi \left(\sum_{i=1}^k d_i^2 \right)|\log \rho | <W_g(u_0). 
\end{equation}
Passing to the limit as \(\rho \to 0\) in the previous equation we obtain \(W_g(u_0)\leq W_g(u_*)<W_g(u_0)\) which is a contradiction.
\end{proof}
\subsection{Neumann boundary conditions}
Analogous results can be stated  for homogeneous Neumann boundary conditions and we leave it to the reader.

\section*{Appendix}

We recall here two lemmas that we use in the proofs of the main results. For the proofs of these lemmas we refer to \cite[Lemma I.3-I.4]{BBH_1994}.

\begin{lemma} \label{lem:elliptic_estimate_1}
Let \(G \subset \R^2\) be a smooth bounded domain, let \(U_i\), \(i=1, \dots,k\) be smooth subdomains of \(G\),  such that \( \O:= G \setminus \cup_{i=1}^k \overline{U}_i\) is connected. Let \(v\) be a function satisfying
\begin{equation}
\left\{
\begin{array}{rcll}
\Delta v&=&0 & \text{ in } \O, \\
\int_{\p U_i} \p_\nu v &=&0 & \text { for each } i=1, \dots,k, \\
\p_\nu v&=&0 & \text{ on } \p G.
\end{array}
\right.
\end{equation}
Then
\begin{equation}
\sup_\Omega v-\inf_\O v \leq \sum_{i=1}^k \left( \sup_{\p U_i} v- \inf_{\p U_i} v \right).
\end{equation}
\end{lemma}

\begin{lemma} \label{lem:elliptic_estimate_2}
Let \(G \subset \R^2\) be a smooth bounded domain, let \(U_i\), \(i=1, \dots,k\) be smooth bounded subdomains of \(G\),  such that \( \O:= G \setminus \cup_{i=1}^k \overline{U}_i\) is connected. Let \(v\) be a function satisfying
\begin{equation}
\left\{
\begin{array}{rcll}
\Delta v&=&0 & \text{ in } \O, \\
\int_{\p U_i} \p_\nu v &=&0 & \text { for each } i=1, \dots,k.
\end{array}
\right.
\end{equation}
Then
\begin{equation}
\sup_\Omega v-\inf_\O v \leq \sum_{i=1}^k \left( \sup_{\p U_i} v- \inf_{\p U_i} v \right)+ \sup_{\p G} v- \inf_{\p G} v.
\end{equation}
\end{lemma}

\bibliographystyle{abbrv}
\bibliography{biblio}
\end{document}